\documentclass{svmult}

\usepackage{makeidx}       
\usepackage{graphicx}     
\usepackage{multicol}       
\usepackage[bottom]{footmisc}

\makeindex             

\input xy
\xyoption{all}
\CompileMatrices

\usepackage{amssymb,amsmath,xy}
\usepackage{color}

\def\U{{\mathcal U}}

\def\R{{\mathbb R}}
\def\T{{\mathbb T}}
\def\t{{\mathfrak{t}}}
\def\H{{\mathcal H}}
\def\K{{\mathcal K}}
\def\W{{\mathcal W}}

\def\Z{{\mathbb Z}}

\def\Br{\operatorname{Br}}

\begin{document}

\title*{Recent advances in the study of the Equivariant Brauer Group}

\author{Peter Bouwknegt\inst{1,2}, Alan Carey\inst{1}\and Rishni Ratnam\inst{1}}

\institute{Mathematical Sciences Institute, Australian National University, Canberra ACT~0200, Australia,
\texttt{peter.bouwknegt@anu.edu.au, alan.carey@anu.edu.au, rishni.ratnam@anu.edu.au}
\and Department of Theoretical Physics, Research School of Physics and Engineering, 
Australian National University, Canberra ACT~0200, Australia} 


\maketitle

\begin{abstract}
In this paper we outline a recent construction of a Chern-Weil isomorphism for the equivariant 
Brauer group of $\R^n$ actions on a principal torus bundle, where the target for this 
isomorphism is a ``dimensionally reduced" \v Cech cohomology group. Using this latter 
group, we demonstrate how to extend the induced algebra construction to algebras 
with a non-trivial bundle as their spectrum.
\end{abstract}

\section{Introduction}

Let $Z$ be a locally compact Hausdorff space.  We fix a locally inner action 
$\operatorname{Ad}u$ of $\Z^n$ on $C_0(Z,\K)$, the C*-algebra of continuous functions
from $Z$ to $\K$, the compact operators on an infinite dimensional, separable, Hilbert space 
$\H$, vanishing at infinity.  I.e.\ $u: m\mapsto u^m \in U(\H)$, and acts by the adjoint 
action on $C_0(Z,\K)$ (recall $\operatorname{Aut}\,\K = PU(\H)$).
Recall the \emph{induced algebra}
\begin{align*}
\operatorname{Ind}_{\Z^n}^{\R^n}&(C_0(Z,\K),\operatorname{Ad}u) =\\ & \{\xi:\R^n
\to C_0(Z,\K):\xi(s+m) =\operatorname{Ad}u^{-m}(\xi(s)), s\in\R^n, m\in\Z^n\}.
\end{align*}
This algebra comes equipped with an action of $\R^n$ given by translations: 
$\alpha'_s(\xi)(t)=\xi(t-s)$. Since $\Z^n$ acts trivially on $Z$, 
\cite[Prop 6.16]{RaeWill98} shows that the spectrum of 
$\operatorname{Ind}_{\Z^n}^{\R^n}(C_0(Z,\K),\operatorname{Ad}u)$ is homeomorphic to 
$\T^n\times Z$, where $\T^n = \R^n/\Z^n$. In the converse direction, 
if $(A,\alpha)\in\mathfrak{Br}_{\R^n}(\T^n\times Z)$, 
then \cite{Ech90} implies there exists a locally inner action of $\Z^n$ on $C_0(Z,\K)$ such that 
$(A,\alpha)$ is $\R^n$-equivariantly isomorphic to $(\operatorname{Ind}_{\Z^n}^{\R^n}(C_0(Z,\K),
\operatorname{Ad}u),\alpha')$. It follows that, for an arbitrary principal $\T^n$-bundle $X$, 
all elements of $\mathfrak{Br}_{\R^n}(X)$ are \emph{locally} $\R^n$-equivariantly 
isomorphic to an induced algebra, but not globally if $\pi:X\to Z$ is a non-trivial bundle. 
The natural question that thus arises is; what is the sufficient data to sew 
together a collection of induced algebras in order to get a 
representative of every element of $\Br_{\R^n}(X)$, i.e.\ pairs in $\mathfrak{Br}_{\R^n}(X)$
up to outer conjugation, where $\pi:X\to \R^n\backslash X$  is a 
nontrivial principal $\R^n/\Z^n$-bundle?

The answer to this question is a cohomology group whose construction was motivated 
by studies into T-duality in String Theory. Let $Z$ be a $C^\infty$ manifold, and denote by 
$H^k_{dR}(Z,\t)$ the $k^{th}$ de Rham cohomology group with values in the Lie algebra 
$\t$ of $\T^n$. Now, suppose $\pi:X\to Z$ is our principal $\R^n/\Z^n$-bundle classified 
over some open cover $\W$ of $Z$ by $[F]\in \check{H}^2(\W,\underline{\Z}^n)$, and 
denote the image of the $F$ in $\Omega^2(Z,\R^n)$ by $F_2$, called a \emph{curvature} 
form of $\pi:X\to Z$.

Let $\wedge^i \t^*$ denote the $i^{th}$ exterior power of the dual Lie algebra 
$\t^*$. For $0\leq m\leq l\leq k$ we define a cochain complex
\begin{equation}\label{BHMComplex}
C_{F_2}^{k,(m,l)}(Z,\t^*):=\bigoplus_{i=m}^l\Omega^{k-i}(Z,\wedge^i \t^*),
\end{equation}
where $\Omega^{k-i}(Z,\wedge^i \t^*)$ is the group of differential $(k-i)$-forms with 
values in $\wedge^i \t^*$. Let $\{X_i\}_{i\in \{1,\dots ,n\} }$ and $\{X_i^*\}_{i\in \{1,\dots ,n\} }$ 
be a basis and dual basis for $\t$, and $\t^*$, respectively. Notice that $\{X_{j_1}^*
\wedge X_{j_2}^* \dots 
\wedge X_{j_i}^*:0\leq j_1<j_2<\dots < j_i\leq n\}$ is a basis for $\wedge^i \t^*$, 
so that every element $H_{(k-i)i}$ of $\Omega^{k-i}(Z,\wedge^i \t^*)$ can be written as a sum
\[\sum_{1\leq j_1<j_2<\dots < j_i\leq n}(\tilde{H}_{(k-i)i})_{j_1\dots j_i}\otimes 
(X_{j_1}^*\wedge\dots\wedge X_{j_i}^*)\,,\]
with $(\tilde{H}_{(k-i)i})_{j_1\dots j_i}\in\Omega^{k-i}(Z)$. To define a differential on 
the cochain complex (\ref{BHMComplex}), we first observe that on each group 
$\Omega^{k-i}(Z,\wedge^i \t^*)$ we can define maps
\[\wedge F_2:\Omega^{k-i}(Z,\wedge^i \t^*)\to \Omega^{k-i+2}(Z,\wedge^{i-1} \t^*)\]
by the formula
\begin{align*}
&H_{(k-i)i}\wedge F_2:=\\
&\sum_{1\leq j_1<j_2<\dots < j_i\leq n}\sum_{l=1}^i(-1)^{l+1}(\tilde{H}_{(k-i)i})_{j_1\dots j_i}
\wedge F_2(X_{j_l}^*)\otimes X_{j_1}^*\wedge\dots \wedge\widehat{X_{j_l}}\wedge
\dots\wedge X_{j_i}^*,\\
\end{align*}
where $\widehat{X_{j_l}}$ denotes omission of $X_{j_l}$. Then we have a differential 
$D_{F_2}$ on (\ref{BHMComplex}) given by
\begin{align*}
&D_{F_2}(H_{(k-m)m},\dots,H_{(k-l)l})=(dH_{(k-m)m}+(-1)^{k-m-1}H_{(k-m-1)(m+1)}\wedge F_2,\\
&\quad dH_{(k-m-1)(m+1)}+(-1)^{k-m-2}H_{(k-m-2)(m+2)}\wedge F_2, \dots,dH_{(k-l)l}),
\end{align*}
and we denote the resulting cohomology groups by $H^{k,(m,l)}_{F_2}(Z,\t^*)$. 
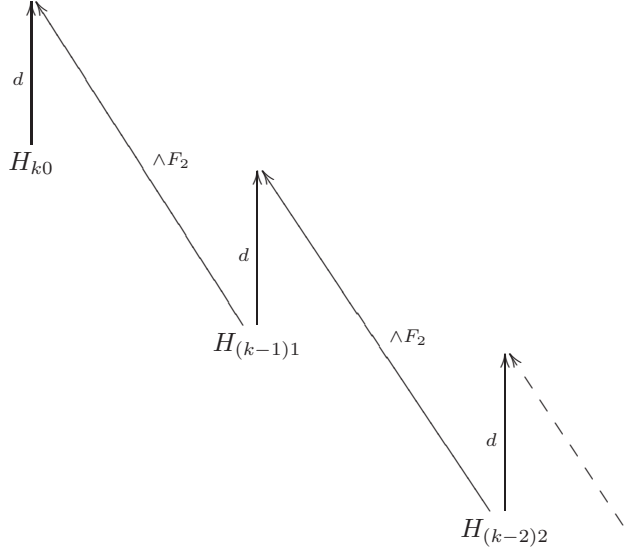
\medskip

\centerline{\xymatrix{
&&&&&\\
&&&&&\\
H_{k0}\ar[uu]^{d}&&&&&\\
&&&&&\\
&&H_{(k-1)1}\ar[uu]^{d}\ar[uuuull]_{\wedge F_2}&&&\\
&&&&&\\
&&&&H_{(k-2)2}\ar[uu]^{d}\ar[uuuull]_{\wedge F_2}& \ar@{-->}[uul]
}} \medskip

\noindent The groups $H^{k,(m,l)}_{F_2}(Z,\t^*)$ are interesting because they fit into 
a ``Gysin" sequence:
\begin{theorem}[\cite{BouHanMat05}]\label{BHMGysin}
There is an exact sequence
\begin{equation}
\to H^{k,(m,m)}_{F_2}(Z,\t^*)\stackrel{\pi^*}{\to} H^{k,(m,l)}_{F_2}(Z,\t^*)\stackrel{\pi_*}{\to} 
H^{k,(m+1,l)}_{F_2}(Z,\t^*)\stackrel{\wedge F_2}{\to} H^{k+1,(m,m)}_{F_2}(Z,\t^*)\to
\end{equation}
where the map $\wedge F_2:H^{k,(m+1,l)}_{F_2}(Z,\t^*)\to H^{k+1,(m,m)}_{F_2}(Z,\t^*)$ is
given by 
\[[(H_{(k-m-1)(m+1)},\dots,H_{(k-l)l})]\wedge F_2:= [(-1)^{k-m}H_{(k-m-1)(m+1)}\wedge F_2].\]
Moreover, there is an isomorphism $H^k_{dR}(X)\cong H^{k,(0,k)}_{F_2}(Z,\t^*)$.
\end{theorem}
In \cite{BouHanMat05} the authors argue that the ``integration over the fibres map" 
$H^3_{dR}(X)\cong H^{3,(0,3)}_{F_2}(Z,\t^*)\stackrel{\pi_*}{\to} H^{3,(1,3)}_{F_2}(Z,\t^*)$ 
computes the T-dual curvature of a pair $(\pi:X\to Z,[H])$, where $[H]\in H^3_{dR}(X)$ is an 
H-flux. For the purposes of this exposition however, what is of interest is the remarkable 
resemblance of Theorem \ref{BHMGysin} with the Gysin sequence of Packer, Raeburn and 
Williams, which we now describe.

Consider the case where $G$ is a second countable locally compact Hausdorff abelian 
group acting on a locally compact space $X$ with orbit space $Z=G\backslash X$. 
Suppose also 
that $N\subset G$ is a closed subgroup such that $G\to N$ and $\hat{G}\to\hat{N}$ 
(where $\hat G$ denotes the Poincar\'e dual to $G$) have 
local sections, and $\pi:X\to Z$ is a principal $G/N$-bundle. For any element 
$[CT(X,\delta),\alpha]\in\Br_G(X)$ there is an open cover $\{U_{\lambda_0}\}$ of $X$ 
and isomorphisms $\Phi_{\lambda_0}:CT(X,\delta)|_{U_{\lambda_0}}\to 
C_0(U_{\lambda_0},\K)$. Moreover, it follows as above that 
$\Phi_{\lambda_0}\circ\alpha|_N\circ\Phi_{\lambda_0}$ is locally inner. 
In \cite{RaeWil93ddclasses} and \cite{PacRaeWill96} the authors develop a cohomology 
theory $H^k_{G}(X,{\mathcal S})$, where $\mathcal S$ denotes the sheaf of germs of 
continuous $\T$-valued functions,  that in degree 2 is isomorphic to the subgroup of 
$\mbox{Br}_G(X)$ such that the restriction of the actions to $N$ are locally unitary. 
The failure of $\Phi_{\lambda_0}\circ\alpha|_N\circ\Phi_{\lambda_0}$ to be locally 
unitary at $x\in X$ is measured by a Mackey obstruction $f\in C(X/G,H_M^2(N,\T))$, 
which is trivial if and only if the restriction of $\alpha$ to $N$ is locally unitary for all 
$x\in X$. (We will describe the Mackey obstruction for principal torus bundles in 
subsection \ref{Mackeysectiononeonetwo} below). Packer, Raeburn and Williams 
call such systems \emph{$N$-principal}.

Let $\U=\{U_{\lambda_0}\}$ be an open cover of $X$ by $G$-invariant sets, and 
define a two column cochain complex $C^{kj}_G(\U,{\mathcal S})$ as follows:
$$
C^{k0}_G(\U,{\mathcal S}):=\check{C}^k(\U,{\mathcal S}),\quad
C^{(k-1)1}_G(\U,{\mathcal S}):= Z_G^1(G,\check{C}^{k-1}(\U,{\mathcal S})),
$$
where $Z_G^1(G,\check{C}^{k-1}(\U,{\mathcal S}))$ denotes the set of \emph{continuous} 
group cohomology 1-\emph{cocycles} from $G$ into the $G$-module 
$\check{C}^{k-1}(\U,{\mathcal S})$ with the obvious $G$ action. In other words, an 
element $\eta\in C^{(k-1)1}_G(\U,{\mathcal S})$ is a collection of continuous functions 
$\eta_{\lambda_0\dots\lambda_{k-1}}: G\times U_{\lambda_0\dots\lambda_{k-1}}\to \T$ 
such that, for all $g_0,g_1\in G$ and $x\in U_{\lambda_0\dots\lambda_{k-1}}$, the 
following holds:
\[\eta_{\lambda_0\dots\lambda_{k-1}}(g_1,x)\eta_{\lambda_0\dots\lambda_{k-1}}
(g_0g_1,x)^*\eta_{\lambda_0\dots\lambda_{k-1}}(g_0,g_1^{-1}x)=1.\]
This complex has as horizontal differential the usual group cohomology differential $\partial_G$, 
and as vertical differential the usual \v{C}ech differential $\check{\partial}$. \medskip

\centerline{\xymatrix{
&\\
\check{C}^2(\U,{\mathcal S})\ar[r]^(0.4){\partial_G}\ar[u]^{\check\partial} & 
Z_G^1(G,\check{C}^{2}(\U,{\mathcal S}))\ar[u]^{\check\partial}\\
\check{C}^1(\U,{\mathcal S})\ar[r]^(0.4){\partial_G}\ar[u]^{\check\partial} & 
Z_G^1(G,\check{C}^{1}(\U,{\mathcal S}))\ar[u]^{\check\partial}\\
\check{C}^0(\U,{\mathcal S})\ar[r]^(0.4){\partial_G}\ar[u]^{\check\partial} & 
Z_G^1(G,\check{C}^{0}(\U,{\mathcal S}))\ar[u]^{\check\partial}
}} \medskip

\noindent The cohomology group $H^k_{G}(\U,{\mathcal S})$ is then the cohomology 
$Z^k_{G}(\U,{\mathcal S})/B^k_{G}(\U,{\mathcal S})$ of the total complex
\[\left(\oplus_{j=0}^1 C^{(k-j)j}_G(\U,{\mathcal S}), \check\partial\oplus (-1)^{k-j}
\partial_G\right).\]
Then, after showing refinement maps induce canonical maps on cohomology 
\cite[Sect.~2]{PacRaeWill96}, the authors define
$H^k_{G}(X,{\mathcal S}):=\varinjlim_{\U}H^k_{G}(\U,{\mathcal S}).$

Consider now the Packer-Raeburn-Williams Gysin sequence:
\begin{theorem}[{\cite[Thm 4.1]{PacRaeWill96}}]\label{PRWGysin}
Let $G$ be a locally compact abelian group and $N$ a closed subgroup 
such that $G\to N$ and $\hat{G}\to\hat{N}$ have local sections. Let ${\mathcal N}$ 
and $\mathcal{\hat{N}}$ denote the sheaves of germs of continuous $N$ and 
$\hat{N}$-valued functions respectively. Then for any principal $G/N$-bundle 
$\pi:X\to Z$ with Euler vector $c\in \check{H}^2(Z,{\mathcal N})$ there is a 
long exact sequence
\[\dots\to \check{H}^k(Z,{\mathcal S})\stackrel{\pi^*_G}{\to} 
H^k_G(X,{\mathcal S})\stackrel{\pi_*}{\to} 
\check{H}^{k-1}(Z,\mathcal{\hat{N}})\stackrel{\cup c}{\to}\check{H}^{k+1}(Z,{\mathcal S})\to\dots\]
The sequence starts with $\check{H}^1(Z,{\mathcal S})$, 
and $\pi^*_G:\check{H}^1(Z,{\mathcal S})\to \check{H}^1_G(X,{\mathcal S})$ is injective.
\end{theorem}
This exact sequence is a \v{C}ech cohomology version of Theorem \ref{BHMGysin} 
in the case that $l=2$. Unlike the proof of Theorem \ref{BHMGysin} however, 
exactness of Theorem \ref{PRWGysin} is a very difficult computation involving 
complicated formulas, in particular at $\check{H}^{k-1}(Z,\mathcal{\hat{N}})$.

The analogue of the conclusion $H^3_{dR}(X)\cong H^{3,(0,3)}_{F_2}(Z,\t^*)$ 
is the following Lemma, which in some sense makes a statement about the integer 
analogue of the group $H^{3,(0,2)}_{F_2}(Z,\t^*)$ (see Theorem \ref{PRWagreement} 
and Theorem \ref{MainSquare}):
\begin{lemma}[{\cite[Lemma 1.3]{PacRaeWill96}}]\label{IsotoKerM}
Suppose that $G$ is a locally compact abelian group, and $N$ a closed subgroup 
such that $\hat{G}\to\hat{N}$ has local sections. Suppose $\pi:X\to Z$ is a locally 
trivial principal $G/N$-bundle over a paracompact space $Z$ and let 
$\operatorname{M}:\operatorname{Br}_G(X)\to C(G\backslash X,H^2_M(N,\mathbb T))$ 
denote the Mackey Obstruction map. 
Then $H_G^2(X,{\mathcal S})\cong \ker \operatorname{M}$.
\end{lemma}

In order to generalise Lemma \ref{IsotoKerM} to all of $\Br_G(X)$ 
(which is the integer analogue of $H^{3,(0,3)}_{F_2}(Z,\t^*)$), one may choose to 
study a three-column complex defined by setting
$C^{k0}_G(\U,{\mathcal S}):=\check{C}^k(\U,{\mathcal S}),$ and
$$
C^{(k-1)1}_G(\U,{\mathcal S}):= C_G^1(G,\check{C}^{k-1}(\U,{\mathcal S})),\quad
C^{(k-2)2}_G(\U,{\mathcal S}):= Z_G^2(G,\check{C}^{k-2}(\U,{\mathcal S}))\,.
$$
However, such a complex is unable to go beyond $\ker\operatorname{M}$, because 
the horizontal differential requires $\U$ to consist of $G$-invariant sets, and 
\cite[Corollary 5.18]{PacRaeWill96} \footnote{The corollary is incorrect as the isomorphism
claimed there is only a surjection.} implies continuous trace algebras trivialisable 
over $G$-invariant sets have trivial Mackey obstruction.

In this exposition, we discuss the appropriate integer \v{C}ech cohomology analogue of 
$H^{3,(0,3)}_{F_2}(Z,\t^*)$. More details can be found elsewhere 
\cite{BouRat1,BouCarRat2}. We then use this cohomology group to present a new 
construction: a representative of every class of $\Br_{\R^n}(X)$ that is \emph{explicitly} 
locally an induced algebra (see Section \ref{UniversalSection}).

\section{Preliminaries}\label{Mackeysectionchapter0}

\subsection[Continuous Trace C*-Algebras]{Continuous Trace C*-Algebras}\label{ctstracesection}

In this paper we are only concerned with separable and stable algebras and so our 
discussion is a greatly restricted version of the usual one for which the reader 
should consult \cite{RaeWill98}.
With $X$ as in the introduction, let $p:E\to X$ be a (locally trivial) vector bundle over 
$X$ with fibre $\K$ (the compact operators
on a separable, infinite dimensional Hilbert space $\H$), and structure group 
$\operatorname{Aut}\K$ (with the point-norm topology).  Let $\Gamma(E,X)$ be the 
$*$-algebra of all continuous sections of $p:E\to X$. Then
\[\Gamma_0(E,X):=\{f\in \Gamma(E,X): x\mapsto ||f(x)|| \mbox{ vanishes at } \infty\}\]
is a $C^*$-algebra with respect to point-wise operations and the sup norm \cite[Prop. 4.89]{RaeWill98}.
\begin{definition}
A separable C*-algebra $A$ with {spectrum} 
 $X$ is called \emph{continuous trace} if it is $C_0(X)$-isomorphic to $\Gamma_0(E,X)$, 
 for some (locally trivial) vector bundle $p: E\to X$ with fibre $\K$, and structure group 
 $\operatorname{Aut}\K$.
\end{definition}

{{} Now, recall the \emph{Brauer group} $\mbox{Br}(X)$ of $X$, which is the group of 
$C_0(X)$-isomorphism classes of continuous trace C*-algebras with spectrum $X$ 
where the zero element is the $C_0(X)$-isomorphism class of $C_0(X,\K)$ and the 
group operation is the balanced tensor product $A\cdot B:= A\otimes_{C_0(X)}B$.
(We are using here the standard abuse of notation by confusing a continuous trace C*-algebra 
$A$ with its $C_0(X)$-isomorphism class). By \cite[Prop 4.53]{RaeWill98}, 
isomorphism classes of vector bundles with fibre $\K$ and structure group 
$\operatorname{Aut}\K$ are in one-to-one correspondence with $\check{H}^1(X,{\mathcal A})$, 
where ${\mathcal A}$ is the sheaf of germs of continuous $\operatorname{Aut}\K$-valued 
functions. If we equip the unitary operators $U(\H)$ with the strong operator topology, 
then there is an exact sequence of topological groups
\begin{equation}\label{Uexactsequence}
1\to \T\to U(\H)\to \operatorname{Aut}\K\to 1,
\end{equation}
such that $U(\H)\to \operatorname{Aut}\K$ has local continuous sections 
\cite[Ch 1]{RaeWill98}. Consequently, the long exact sequence in sheaf 
cohomology, combined with the fact that $U(\H)$ is contractible 
(in the strong operator topology) \cite[Thm 4.72]{RaeWill98} implies that 
$\check{H}^1(X,{{\mathcal A}})\cong \check{H}^2(X,{{\mathcal S}})$.
{}From the exact sequence $0\to \Z\to \R\to \T\to 1,$
we then obtain
$\check{H}^1(X,{{\mathcal A}})\cong \check{H}^2(X,{{\mathcal S}})\cong 
\check{H}^3(X,\underline{\Z}).$
As continuous trace C*-algebras are $C_0(X)$-isomorphic if and only if they are 
$C_0(X)$-isomorphic to the algebra of sections of isomorphic bundles, 
we obtain the \emph{Dixmier-Douady} classification \cite{DixDou63}.
Namely,
if $A$ is a continuous trace C*-algebra with spectrum $X$, $C_0(X)$-isomorphic to
continuous sections of a bundle $p:E\to X$, then there is an isomorphism 
$\mathrm{Br}(X)\to \check{H}^3(X,\underline{\Z})$ defined by the map induced by sending $A$ 
to the image of $p:E\to X$ in $\check{H}^3(X,\underline{\Z})$.
The corresponding unique (up to $C_0(X)$-isomorphism) continuous trace 
C*-algebras with image $\delta\in\check{H}^3(X,\underline{\Z})$ will be denoted 
$CT(X,\delta)$. We now move on to the \emph{equivariant} Brauer group. 

\begin{definition}
Fix a locally compact group $G$ and a second countable locally compact Hausdorff (left) 
$G$-space $X$. Denote the induced $G$-action on $C_b(X)$ by $\tau$. That is
$\tau_g(f)(x)=f(g^{-1}x),$ for $ f\in C_b(X).$
Let $A$ be a continuous trace C*-algebra with spectrum $X$, and $\alpha$ an action 
of $G$ on $A$. Then $\alpha$ is said to \emph{preserve the (given) action on the spectrum} 
if for all $g\in G$, $a\in A$ and $f\in C_b(X)$:
$\alpha_g(f\cdot a)=\tau_g(f)\cdot \alpha_g(a).$
\end{definition}
We write $\mathfrak{Br}_G(X)$ for the collection of pairs $(A,\alpha)$, where $A$ is a 
continuous trace C*-algebra with spectrum a $G$-space $X$ and $\alpha$ is an action 
of $G$ on $A$ that preserves the given action on the spectrum.
If $A$ is a C*-algebra, we denote by $M(A)$ and $UM(A)$ the \emph{multiplier algebra} 
of $A$ and unitary elements of $M(A)$ respectively. 
Now, given elements $(A,\alpha)$ and $(A,\beta)$ of $\mathfrak{Br}_G(X)$, we say 
the actions $\alpha$ and $\beta$ are \emph{exterior equivalent} if there is a (strictly) 
continuous map $w:G\to UM(A)$ such that
\begin{align}
\label{exteriorcondition1}\beta_g(a)&=w_g\alpha_g(a)w_g^* \quad \mbox{ for all } 
a\in A \mbox{ and } g\in G,
\\
\label{exteriorcondition2}w_{gh}&=w_g\overline{\alpha}_g(w_h)\quad \mbox{ for all } g,h\in G.
\end{align}
In this case, one says that $w$ is a \emph{unitary} $\alpha$ \emph{cocycle} implementing 
the equivalence. Two pairs $(A,\alpha)$ and $(B,\beta)$ are \emph{outer conjugate} if 
and only if there is a $C_0(X)$-isomorphism $\phi:A\to B$ such that $\alpha$ and 
$\phi^{-1}\circ \beta\circ \phi$ are exterior equivalent.
 {{} With the notation of the previous definition we have:
\begin{definition}
The \emph{equivariant Brauer group} $\operatorname{Br}_G(X)$  is $\mathfrak{Br}_G(X)$
modulo outer conjugacy. The zero element is the equivalence class of 
$(C_0(X,\K),\tau)$,  the binary operation is
$[A,\alpha]\cdot[B,\beta]=[A\otimes_{C_0(X)}B,\alpha\otimes_X \beta]$ and
 the inverse of $[A,\alpha]$ is
the conjugate algebra $[\overline{A},\overline{\alpha}]$ (we use the canonical
bijection of sets $\flat:A\to \overline{A}$, to define $\overline{\alpha}_g(\flat(a)):=\flat(\alpha_g(a))$.
\end{definition}
Later on we will use the ``forgetful homomorphism" $F:\mbox{Br}_G(X)\to \mbox{Br}(X)$ 
that sends $[CT(X,\delta),\alpha]$ to $[CT(X,\delta)]$. This map is neither surjective 
nor injective in general.

\subsection{$G/N$-bundles}\label{bundles}

Let $G$ be a locally compact topological group acting on the left of a locally compact space 
$X$, with orbit space $Z=G\backslash X$. Suppose further that $N\subset G$ is a closed 
subgroup that acts trivially on $X$, and that $\pi:X\to Z$ is a principal $G/N$-bundle. Here 
we assume, and we shall do so for the rest of this article, that all bundles are \emph{locally trivial}. 
Local triviality means there is an open cover $\W=\{W_i\}$ of $Z$ and there 
exist continuous local sections 
$\sigma_i:W_i\to\pi^{-1}(W_i)$. Moreover, since $X$ is a \emph{principal} bundle, the action of 
$G/N$ is free (\cite[Example 4.60]{RaeWill98}), and the local sections define continuous 
$G/N$-equivariant maps $w_i:\pi^{-1}(W_i)\to G/N$ by
\[w_i(x)^{-1}x=\sigma_i(\pi(x)), \quad x\in \pi^{-1}(W_i).\]
In turn, these equivariant maps give local $G/N$-equivariant homeomorphisms 
$h_i:\pi^{-1}(W_i)\to G/N\times W_i$ defined by
\[h_i(x):=(w_i(x),\pi(x)),\]
which we call \emph{local trivialisations}. Employing the notation $W_{ij}:=W_i\cap W_j$, 
the local trivialisations allow us to define so-called \emph{transition functions} 
$t_{ij}:W_{ij}\to G/N$ by
\[h_j\circ h_i^{-1}([s]_{G/N},z)=(t_{ij}(z)[s]_{G/N},z),\quad z\in Z, [s]_{G/N}\in G/N,\]
and one can check they satisfy the two identities
\begin{align*}
&t_{ij}(\pi(x))=w_i(x)^{-1}w_j(x), \quad \mbox{ and}\\
&t_{ij}(z)\sigma_j(z)=\sigma_i(z).
\end{align*}
The transition functions define an element $[t]$ in the first \v{C}ech cohomology 
group $\check{H}^1(Z,\mathcal{G}/\mathcal{N})$, where $\mathcal{G}/\mathcal{N}$ 
is the sheaf of germs of continuous $G/N$-valued functions, and the 
map $[X,\pi,Z]\mapsto [t]$ defines a bijective correspondence between isomorphism 
classes of principal $G/N$-bundles and $\check{H}^1(Z,\mathcal{G}/\mathcal{N})$ 
(see, e.g., \cite[Prop 4.53]{RaeWill98}).

If the projection $G\to G/N$ has local sections we can pick (perhaps after a refinement 
of the cover) local lifts $s_{ij}:W_{ij}\to G$ that satisfy, for all $x\in W_{ij}$,  
$[s_{ij}(x)]_{G/N}=t_{ij}(x)$, and define $[F]\in \check{H}^2(\mathcal{W},\mathcal{N})$ 
by $F_{ijk}:=s_{jk}s_{ij}s_{ik}^{-1}$. The hypothesis of the existence of local lifts holds 
for many cases of interest, such as whenever $G=\R^k\times \T^m\times \Z^n\times F$, 
for some finite group $F$, regardless of $N$ (cf. \cite{PacRaeWill96}). In particular, 
when $G=\R^n$, $N=\Z^n$, the fact that $\mathcal{R}^n$ (the sheaf of germs of 
continuous $\R^n$-valued functions) is ``fine" implies for any $k\geq 1$ that 
$\check{H}^k(Z,\mathcal{R}^n)\cong 0$, and the long exact sequence in sheaf 
cohomology implies there is an isomorphism
\[ \Delta:\check{H}^1(Z,\mathcal{G}/\mathcal{N})\stackrel{\cong}{\to} \check{H}^2(Z,\underline{\Z}^n).\]

\begin{definition}
We shall call $ \Delta[t] \in\check{H}^2(Z,\underline{\Z}^n)$ the \emph{Euler vector} 
of the principal $\R^n/\Z^n$-bundle.
\end{definition}

\subsection{Mackey Obstruction for Principal $\T^n$-bundles}\label{Mackeysectiononeonetwo}

Fix an action $\alpha$ of $\R^n$ on on $CT(X,\delta)$ which covers the fibre action 
for a principal $\T^n$-bundle $\pi:X\to Z$. This data gives an element $[CT(X,\delta),\alpha]$ 
of the equivariant Brauer group $\Br_{\R^n}(X)$. A commonly used obstruction of the action 
$\alpha$ is the \emph{Mackey obstruction}. The Mackey obstruction 
$f\in C(\R^n\backslash X,H_M^2(\Z^n,\T))$ 
is trivial if and only if the restriction of $\alpha$ to $\Z^n$ is locally unitary for all $x\in X$. 
Here we give a simple formula. In the work 
below, the symbol $M_n^u(\T)$ denotes the group, under addition, of strictly upper triangular 
$n\times n$ matrices with entries in $\T$.

Choose an open cover $\{U_{\lambda_0}\}_{\lambda_0\in \mathcal{I}}$ of $X$ so that there 
are isomorphisms
$\Phi_{\lambda_0}:CT(X,\delta)\vline_{U_{\lambda_0}}\stackrel{\cong}{\longrightarrow}
 C_0(U_{\lambda_0},\K).$ Since $\Z^n$ acts trivially on the spectrum $X$, 
\cite[Prop 2.1]{EchNes01} implies the actions $\Phi_{\lambda_0}\circ\alpha|_{\Z^n}\circ
\Phi^{-1}_{\lambda_0}$ on $C_0(U_{\lambda_0},\K)$ are locally inner. 
Perhaps after a refinement of $\{U_{\lambda_0}\}$, there exist $u_{\lambda_0}^{i}\in
 C(U_{\lambda_0},U(\H))$ that implement the actions:
\begin{equation}\label{chapter0eq0}
\Phi_{\lambda_0}\circ\alpha_{e_i}\circ\Phi^{-1}_{\lambda_0}(h)
(x)=u_{\lambda_0}^i(x)h(x)u_{\lambda_0}^i(x)^* \quad h\in C_0(U_{\lambda_0},\K) \,,
\end{equation}
where $e_i$ is the $i$-th basis element of $\mathbb Z^n$.
Now define a function $f_{\lambda_0}\in C(U_{\lambda_0}, M_n^u(\T))$ by
\begin{equation}\label{MackeyDefn}
f_{\lambda_0}(x)_{ij}u_{\lambda_0}^i(x)u_{\lambda_0}^j(x)=u_{\lambda_0}^j(x)u_{\lambda_0}^i(x)\,,
\qquad i\leq j\,.
\end{equation}
One can prove that that $f_{\lambda_0}$ is in fact independent of $\lambda_0$, and 
constant on orbits of $X$, and thus defines a global function $f\in C(Z,M_n^u(\T))$. 
The function $f$ is called the Mackey obstruction of $\alpha$.

\subsection{Phillips-Raeburn Obstruction}\label{phillipsraeburnsection}

Let $\hat{\Z}^n$ denote the Pontryagin dual of $\Z^n$, and $\mathcal{\hat{N}}$ the 
sheaf of germs of continuous $\hat{\Z}^n$-valued functions. Suppose that we have 
$[C_0(X,\K),\alpha]\in \Br_{\R^n}(X)$, as above, and $\alpha$ has vanishing Mackey 
obstruction. Then there exists an open cover $\U=\{U_{\lambda_0}\}_{\lambda_0\in 
\mathcal{I}}$ of $X$ and \emph{homomorphisms} $u_{\lambda_0}:\Z^n\to 
C(U_{\lambda_0},U(\H))$ such that
\[(\alpha_{m}h)(x)=u_{\lambda_0}^m(x)h(x)u_{\lambda_0}^m(x)^*,\quad h\in C_0(X,\K)\,,\qquad
m\in\mathbb Z^n\,. \]
In this case the action $\alpha$ is called \emph{locally unitary}. One can then 
measure the obstruction to $\alpha$ being implemented by a global homomorphism
$u:\Z^n\to C(X,U(\H))$ via the cocycle $\eta_{\lambda_0\lambda_1}\in \check{Z}^1
(\U,\mathcal{\hat{N}})$ defined by
\begin{equation}\label{Philrae}
\eta_{\lambda_0\lambda_1}^m(x):= u_{\lambda_1}^m(x)
u_{\lambda_0}^m(x)^*,\quad x\in U_{\lambda_0\lambda_1}.
\end{equation}
The class $[\eta]$ is called the \emph{Phillips-Raeburn obstruction} of $\alpha$ 
\cite{PhiRae79}, and clearly $\alpha$ can be implemented globally if and only 
if $[\eta]=0\in \check{H}^1(X,\mathcal{\hat{N}})$.

\section{Dimensionally Reduced Cohomology}\label{DimRedCohoSection}

Here we define a series of  ``dimensionally reduced cohomology" groups in order to 
generalise the dimensionally reduced de Rham cohomology groups from 
\cite{BouHanMat05} and to extend Lemma \ref{IsotoKerM} to all of $\Br_G(X)$. 
{}From here on we specialize the discussion to $G=\mathbb R^n$, and $N=\mathbb Z^n$.

Let $Z$ be a $C^\infty$ manifold and fix an open cover ${{\mathcal W}}=\{W_{\mu_0}\}$ of 
$Z$ together with a cocycle $F\in \check{Z}^2({{\mathcal W}},\underline{\Z}^n)$. 
Let $\mathcal{S}, \hat{\mathcal{N}}$ and $\mathcal{M}$ denote the sheaves of 
germs of continuous $\T,\hat{\Z}^n$ and $M_n^u(\T)$-valued functions respectively. 
We can think of a \v{C}ech cocycle in $\phi^{(k-2)2}\in \check{C}^{k-2}(\W,{\mathcal M})$ 
as an $n \choose 2$-tuple $\{\phi^{(k-2)2}(\cdot)_{ij}\}_{1\leq i<j\leq n},$
where $\phi^{(k-2)2}(\cdot)_{ij}\in \check{C}^{k-2}(\W,\mathcal{S})$.

We define a cochain complex
$C^k_F({{\mathcal W}},{{\mathcal S}}),$
for $k\geq 2$ as all triples 
\begin{equation*}
(\phi^{k0},\phi^{(k-1)1},\phi^{(k-2)2})\,,
\end{equation*}
consisting of \v{C}ech cochains 
$\phi^{k0}\in \check{C}^{k}(\W,\mathcal{S})$, $\phi^{(k-1)1}\in \check{C}^{k-1}(\W,\hat{\mathcal{N}})$ and 
$\phi^{(k-2)2}\in \check{C}^{k-2}(\W,\mathcal{M})$. When $k=1$, we define a cochain to be a pair 
$(\phi^{10},\phi^{01
})$, where $\phi^{10}\in \check{C}^{1}(\W,\mathcal{S})$, $\phi^{01}\in \check{C}^{0}(\W,\hat{\mathcal{N}})$, 
whilst when $k=0$ a cochain is an element  $\phi^{00}\in \check{C}^{0}(\W,\mathcal{S})$.

Now, for any $A\in \check{C}^2(\W,\underline{\Z}^n)$ and $B\in \check{C}^3(\W,\underline{M_n^u(\Z)})$ 
we can define products
\begin{align*}
\cup_1 A: &\ \check{C}^{k-1}(\W,\hat{\mathcal{N}})\to \check{C}^{k+1}(\W,\mathcal{S})\,,\\
\cup_1 A: & \  \check{C}^{k-2}(\W,\mathcal{M}) \to \check{C}^{k}(\W,\hat{\mathcal{N}})\,,\\
\cup_2 B: & \ \check{C}^{k-2}(\W,\mathcal{M}) \to \check{C}^{k+1}(\W,\mathcal{S})\,,
\end{align*}
by the formulas
\begin{align*}
(\phi^{(k-1)1} & \cup_1 A)_{\lambda_0\dots\lambda_{k+1}}(z):= 
\phi^{(k-1)1}_{\lambda_0\dots\lambda_{k-1}}({A_{\lambda_{k-1}\lambda_k\lambda_{k+1}}(z),z)}\,,
\nonumber \\
(\phi^{(k-2)2}& \cup_1 A)_{\lambda_0\dots\lambda_{k+1}}(m,z):= \nonumber  \\ & \prod_{1\leq i<j\leq n}
\phi^{(k-2)2}_{\lambda_0\dots\lambda_{k-1}}(z)_{ij}^{A_{\lambda_{k-1}\lambda_k
\lambda_{k+1}}(z)_i(m)_j-(m)_iA_{\lambda_{k-1}\lambda_k\lambda_{k+1}}(z)_j},
\nonumber \\
(\phi^{(k-2)2} & \cup_2 B)_{\lambda_0\dots\lambda_{k+1}}(z):=\prod_{1\leq i<j\leq n}
\phi^{(k-2)2}_{\lambda_0\dots\lambda_{k-2}}(z)_{ij}^{B_{\lambda_{k-2}\lambda_{k-1}
\lambda_{k}\lambda_{k+1}}(z)_{ij}}.
\end{align*}
where $m_l$ denotes the $l^{th}$ component of $m\in\Z^n$.
The \v{C}ech differential $\check\partial$ is a graded derivation with respect to these 
products \cite{BouRat1}.
Now, let $F_i$ denote the $i^{th}$ component of $F$, our fixed representative of the 
Euler vector of $\pi:X\to Z$, and define a 3-cochain $C(F)\in \check{C}^3(\W,\underline{M_n^u(\Z)})$ 
by the formula:
\[C(F)_{\lambda_0\lambda_1\lambda_2\lambda_3}(z)_{ij}:=F_{\lambda_0\lambda_1\lambda_2}(z)_i
F_{\lambda_0\lambda_2\lambda_3}(z)_j-F_{\lambda_1\lambda_2\lambda_3}(z)_i
F_{\lambda_0\lambda_1\lambda_3}(z)_j\,.\]
By \cite[Sect 2]{Ste47}, $C(F)$ has the property that $\check\partial C(F)_{ij}=F_i\cup F_j-F_j\cup F_i$. 
This property implies the map $D_F: C^k_F({\mathcal W},{\mathcal S})\to C^{k+1}_F({\mathcal W},
{\mathcal S})$ defined by
\begin{align}  \label{dfdefn}
D_F & (\phi^{k0}, \phi^{(k-1)1},\phi^{(k-2)2}):=
(\check\partial\phi^{k0}\times(\phi^{(k-1)1}\cup_1 F)^{(-1)^{k+1}} \nonumber \\ & 
\times(\phi^{(k-2)2}\cup_2 C(F)))^{(-1)^{k+1}},
\quad\check\partial\phi^{(k-1)1}\times(\phi^{(k-2)2}\cup_1 F)^{(-1)^{k}},\check\partial\phi^{(k-2)2})\notag.
\end{align}
satisfies $D_F^2=0$. \medskip

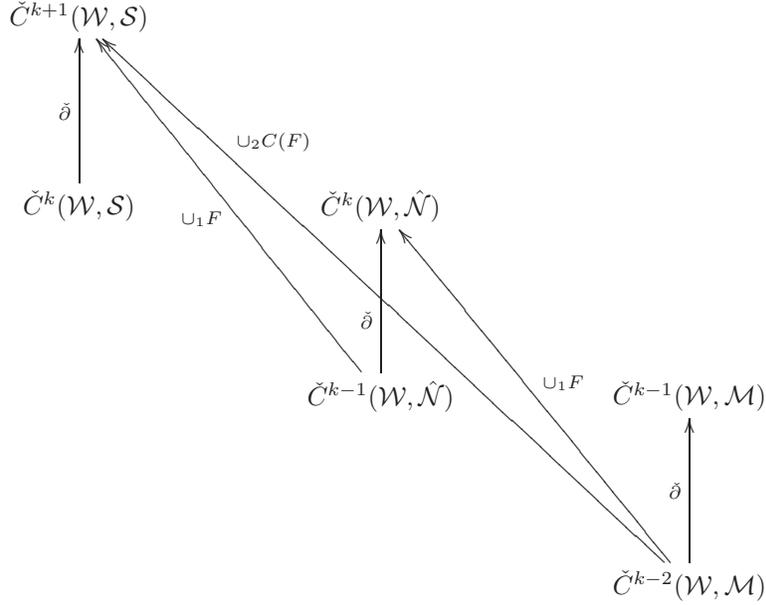
\centerline{\xymatrix{
\check{C}^{k+1}(\W,\mathcal{S})&&&\\
&&&\\
\check{C}^{k}(\W,\mathcal{S})\ar[uu]^{\check\partial}&&\check{C}^{k}(\W,\hat{\mathcal{N}})&&\\
&&\\
&&\check{C}^{k-1}(\W,\hat{\mathcal{N}})\ar[uu]^(0.4){\check\partial}\ar[uuuull]^{\cup_1 F}&&
\check{C}^{k-1}(\W,\mathcal{M})\\
&&&&\\
&&&&\check{C}^{k-2}(\W,\mathcal{M})\ar[uu]^{\check\partial}\ar[uuuull]_{\cup_1 F}
\ar[uuuuuullll]_(0.75){\cup_2 C(F)}
}}\medskip

\noindent Thus we have:
\begin{definition}
The $k^{th}$-dimensionally reduced \v{C}ech cohomology group of the covering $\W$ 
with coefficients in ${\mathcal S}$,
is the cohomology of $C^k_F({{\mathcal W}},{{\mathcal S}})$  under the differential $D_F$. 
 This group is denoted  
${{\mathbb H}}^k_F(\W,{\mathcal S}).$
\end{definition}
We can also define similarly, ${\mathbb H}^k_F(\W,\underline{\Z})$ 
and ${\mathbb H}^k_F(\W,{\mathcal  R})$, 
using integer  and real coefficients.
 Cochains in $C^k_F(\W,\underline{\Z})$ are triples $(\phi^{k0},\phi^{(k-1)1},\phi^{(k-2)2})$ 
 consisting of  \v{C}ech cochains $\phi^{k0}\in \check{C}^{k}(\W,\underline{\Z})$, $\phi^{(k-1)1}
 \in \check{C}^{k-1}(\W,\underline{\Z}^n)$ and $\phi^{(k-2)2}\in \check{C}^{k-2}(\W,\underline{M_n^u(\Z)})$. 
We define degree 0 and 1 cochains as before, by truncating the lower \v{C}ech cochains. 
Then we have maps
\begin{align*}
\cup_1 F:  &\  \check{C}^{k-1}(\W,\underline{\Z}^n)\to \check{C}^{k+1}(\W,\underline{\Z})\,,\\
\cup_1 F:  & \  \check{C}^{k-2}(\W,\underline{M_n^u(\Z)}) \to \check{C}^{k}(\W,\underline{\Z}^n)\,, \\
\cup_2 C(F): &\ \check{C}^{k-2}(\W,\underline{M_n^u(\Z)}) \to \check{C}^{k+1}(\W,\underline{\Z})\,,
\end{align*}
with their integer cohomology analogues:
\begin{align*}
(\phi^{(k-1)1} & \cup_1 F)_{\lambda_0\dots\lambda_{k+1}}(z) :=
\sum_{l=1}^n\phi^{(k-1)1}_{\lambda_0\dots\lambda_{k-1}}(z)_l F_{\lambda_{k-1}
\lambda_k\lambda_{k+1}}(z)_l,\\
(\phi^{(k-2)2}& \cup_1 F)_{\lambda_0\dots\lambda_{k}}(z)_l := \\
& \sum_{1\leq i<j\leq n}\phi^{(k-2)2}_{\lambda_0\dots\lambda_{k-2}}(z)_{ij}    
(F_{\lambda_{k-2}\lambda_{k-1}\lambda_{k}}(z)_i(e_l)_j-(e_l)_iF_{\lambda_{k-2}\lambda_{k-1}
\lambda_{k}}(z)_j),\\
(\phi^{(k-2)2}& \cup_2 C(F))_{\lambda_0\dots\lambda_{k+1}}(z):=
\sum_{1\leq i<j\leq n}\phi^{(k-2)2}_{\lambda_0\dots\lambda_{k-2}}(z)_{ij} 
C(F)_{\lambda_{k-2}\lambda_{k-1}\lambda_{k}\lambda_{k+1}}(z)_{ij},
\end{align*}
Then the differential in integer coefficients is
\begin{align*}
D_F & (\phi^{k0}, \phi^{(k-1)1},\phi^{(k-2)2}):=
(\check\partial\phi^{k0}+(-1)^{k+1}\phi^{(k-1)1}\cup_1 F \\& +(-1)^{k+1}\phi^{(k-2)2}\cup_2 C(F),
\check\partial\phi^{(k-1)1}+(-1)^{k}\phi^{(k-2)2}\cup_1 F,\check\partial\phi^{(k-2)2}) \,.
\end{align*}

\begin{definition}
Fix a cocycle $F\in\check{Z}^2(\W,\underline{\Z}^n)$. The \emph{$k^{th}$ dimensionally reduced 
\v{C}ech cohomology group of the cover $\W$ with coefficients in $\Z$} is the cohomology of 
$C^k_F({\mathcal W},\underline{\Z})$.
\end{definition}
{{}
Similarly, for real coefficients, cochains in $C^k_F(\W,{\mathcal  R})$ are 
triples \[(\phi^{k0},\phi^{(k-1)1},\phi^{(k-2)2})\] consisting of  \v{C}ech cochains 
$\phi^{k0}\in \check{C}^{k}(\W,{\mathcal  R})$, $\phi^{(k-1)1}\in \check{C}^{k-1}(\W,{\mathcal  R}^n)$ 
and $\phi^{(k-2)2}\in \check{C}^{k-2}(\W,\mathcal{M}({\mathcal  R})),$ where ${\mathcal  R}$ denotes 
the sheaf of germs of continuous $\R$-valued functions, and $\mathcal{M}({\mathcal  R})$, the sheaf 
of germs of continuous $M_n^u(\R)$-valued functions.}

As expected, the short exact sequence of groups $0\to \Z\to\R\to\T\to 1$ induces a long 
exact sequence in dimensionally reduced cohomology. As with ordinary \v{C}ech cohomology, 
most of the groups ${\mathbb H}^k_F(\W,{\mathcal  R})$ are trivial. Computing the non-trivial 
ones using a contracting homotopy gives
{{}
\begin{corollary}\label{dimreducedlongexact}
Let $\W$ be a good open cover of a $C^\infty$ manifold $Z$, and fix a cocycle 
$F\in \check{Z}^2(\W,\underline{\Z}^n)$. Then we have exact sequences
\begin{align*}
0&\to C(Z,\Z)\to C(Z,\R)\to C(Z,\T)\to  {\mathbb H}^1_{F}(\W,\underline{\Z})
\to C(Z,\R^n)\to {\mathbb H}^1_{F}(\W,{\mathcal S})\\ & \quad \to {\mathbb H}^2_{F}(\W,\underline{\Z})
 \to C(Z,M_n^u(\R))\to {\mathbb H}^2_{F}(\W,{\mathcal S})\to {\mathbb H}^3_{F}(\W,\underline{\Z})\to 0,
\end{align*}
and
$\quad 0\to {\mathbb H}^k_{F}(\W,{\mathcal S})\to 
{\mathbb H}^{k+1}_{F}(\W,\underline{\Z})\to 0,\quad k\geq 3.$
\end{corollary}

In order to mimic Theorem \ref{PRWGysin}, we introduce a generalisation of 
$\check{H}^{k-1}(Z,\mathcal{\hat{N}})$. Define a cochain complex
$\overline{C}^k_F(\mathcal{W},\mathcal{S})$, 
where for $k\geq 1$ an element is a pair $(\varphi^{k1},\varphi^{(k-1)2})$ consisting of 
\v{C}ech cochains $\phi^{k1}\in \check{C}^{k}(\W,\hat{\mathcal N})$ and $\phi^{(k-1)2}\in 
\check{C}^{k-1}(\W,\mathcal{M})$. A cochain in $\overline{C}^0_F(\mathcal{W},\mathcal{S})$ 
is given by $(\varphi^{01})$, for $\phi^{01}\in \check{C}^{0}(\W,\hat{\mathcal N})$. This complex 
has a differential $\overline{D}_F$ given by
\begin{equation}
\label{overdfdefn}\overline{D}_F(\phi^{k1},\phi^{(k-1)2}):=(\check\partial\phi^{(k-1)1}
\times(\phi^{(k-2)2}\cup_1 F)^{(-1)^{k}},\check\partial\phi^{(k-2)2}).
\end{equation}
\begin{definition}
We define $\overline{\mathbb{H}}^k_F(\mathcal{W},\mathcal{S})$ to be the cohomology of 
$\overline{C}^k_F(\mathcal{W},\mathcal{S})$  under the differential $\overline{D}_F$.
\end{definition}
One can see this group is obtained by removing the first entry from 
$C^{k+1}_{F}(\W,\mathcal{S})$. Obviously, one can define the 
same groups with integer and real coefficients, denoted 
$\overline{\mathbb{H}}^k_F(\mathcal{W},\underline{\Z})$ and 
$\overline{\mathbb{H}}^k_F(\mathcal{W},\mathcal{R})$ respectively. If $\W$ is good, 
then we can define maps in analogue to the Gysin sequence maps from 
Theorem \ref{PRWGysin}. Indeed, we define
\begin{align*}
\pi^*: &\check{H}^{k}(\W,\mathcal{S})\to \mathbb{H}^{k}_{F}(\W,\mathcal{S})\\
&[\phi^{k0}]\mapsto [\phi^{k0},1,1]\\
\pi_*: &\mathbb{H}^{k}_{F}(\W,\mathcal{S})\to \overline{\mathbb{H}}^{k-1}_{F}(\W,\mathcal{S})\\
&[\phi^{k0},\phi^{(k-1)1},\phi^{(k-2)2}]\mapsto [\phi^{(k-1)1},\phi^{(k-2)2}]\\
\cup F:&\overline{\mathbb{H}}^{k-1}_{F}(\W,\mathcal{S})\to \check{H}^{k+1}(\W,\mathcal{S})\\
&[\phi^{(k-1)1},\phi^{(k-2)2}]\mapsto [(\phi^{(k-1)1}\cup_1 F)^{(-1)^{k+2}
}\times(\phi^{(k-2)2}\cup_2 C(F))^{(-1)^{k+2}}].
\end{align*}
Notice that the map $\cup F$ is defined using the fixed representative 
$F\in \check{Z}^2(Z,\underline{\Z}^n)$. That $\W$ is good ensures that the above 
maps commute with the long exact sequence maps from Corollary \ref{dimreducedlongexact} 
(and their analogue for $\overline{\mathbb{H}}^k_F(\W,\bullet)$). 
We then have our analogue of Theorem \ref{PRWGysin}:

\begin{theorem}\label{MyGysin}
Let $\W$ be a good open cover of a $C^\infty$ manifold $Z$, and fix a cocycle 
$F\in \check{Z}^2(\W,\underline{\Z}^n)$. Then there is a commuting diagram with exact columns and rows:

\centerline{\xymatrix{
\ar[r]&\check{H}^k(\W,\mathcal{R})\ar[r]^{\pi^*}\ar[d]& \mathbb{H}^k_F(\W,\mathcal{R})\ar[r]^{\pi_*}\ar[d]&
 \overline{\mathbb{H}}^{k-1}_{F}(\W,\mathcal{R})\ar[d]\ar[r]^{\cup F}&
 \check{H}^{k+1}(\W,\mathcal{R})\ar[d]\ar[r]&\\
\ar[r]&\check{H}^k(\W,\mathcal{S})\ar[r]^{\pi^*}\ar[d]& \mathbb{H}^k_F(\W,\mathcal{S})\ar[r]^{\pi_*}\ar[d]& 
\overline{\mathbb{H}}^{k-1}_{F}(\W,\mathcal{S})\ar[d]\ar[r]^{\cup F}&\check{H}^{k+1}(\W,\mathcal{S})
\ar[d]\ar[r]&\\
\ar[r]&\check{H}^{k+1}(\W,\underline{\Z})\ar[r]^{\pi^*}& \mathbb{H}^{k+1}_F(\W,\underline{\Z})\ar[r]^{\pi_*}& 
\overline{\mathbb{H}}^{k}_{F}(\W,\underline{\Z})\ar[r]^{\cup F} & \check{H}^{k+2}(\W,\underline{\Z})\ar[r]&\\
}}
\end{theorem}

The proof of the above theorem is routine following exactly from the definitions. 
This is quite distinct from the proof of Theorem \ref{PRWGysin}, which as we have already 
mentioned, is very difficult. As desired, this theorem agrees with Theorem 
\ref{PRWGysin} in the following sense:

\begin{theorem}\label{PRWagreement}
Let $\pi:X\to Z$ be a principal $\R^n/\Z^n$-bundle over a $C^\infty$ manifold $Z$, 
and let $F\in\check{Z}^2(\W,\underline{\Z}^n)$ be a representative of the Euler vector 
defined over a good open cover $\W$ of $Z$. Then there is a commutative diagram with exact rows:

\centerline{\xymatrix{
\ar[r]&\check{H}^k(\W,{\mathcal S})\ar[r]^{\pi^*_{\R^n}}\ar[d]^{\operatorname{id}}& 
H^k_{\R^n}(\pi^{-1}(\W),{\mathcal S})\ar[r]^{\pi_*}\ar[d]&
\check{H}^{k-1}(\W,\hat{{\mathcal N}})\ar[r]^{\cup F}\ar[d]&
\check{H}^{k+1}(\W,{\mathcal S})\ar[r]\ar[d]^{\operatorname{id}}&\\
\ar[r]&\check{H}^{k}(\W,{\mathcal S})\ar[r]^{\pi^*}& {\mathbb H}^{k}_F(\W,{\mathcal S})
\ar[r]^{\pi_*}& \overline{{\mathbb H}}^{k-1}_{F}(\W,{\mathcal S})\ar[r]^{\cup F}& 
\check{H}^{k+1}(\W,{\mathcal S})\ar[r]&
}}

\end{theorem}
To prove this theorem we need only define the downward arrows 
$H^k_{\R^n}(\pi^{-1}(\W),{\mathcal S})\to {\mathbb H}^{k}_F(\W,{\mathcal S})$ 
and $\check{H}^{k-1}(\W,\hat{{\mathcal N}})\to \overline{{\mathbb H}}^k_F(\W,{\mathcal S})$, 
and the commutativity will follow immediately from the definitions. Indeed, if $(\nu,\eta)$ 
is a cochain in $C^k_{\R^n}(\pi^{-1}(\W),{\mathcal S})$, its image in $C^{k}_F(\W,{\mathcal S})$ 
is defined to be the triple 
\begin{equation*}
(\phi^{k0}(\nu,\eta),\phi^{(k-1)1}(\nu,\eta),\phi^{(k-2)2}(\nu,\eta))
\end{equation*} 
given by
\begin{align*}
\phi(\nu,\eta)^{k0}_{\mu_0\dots\mu_k}(z):=&\nu_{\mu_0\dots\mu_k}(\sigma_{\mu_k}(z))
\eta_{\mu_0\dots\mu_{k-1}}(-s_{\mu_{k-1}\mu_k}(z),\sigma_{\mu_k}(z)),\\
\phi(\nu,\eta)^{(k-1)1}_{\mu_0\dots\mu_{k-1}}(m,z):=&\eta_{\mu_0\dots\mu_{k-1}}(m,\sigma_{\mu_{k-1}}(z)),
\quad\mbox{and}\quad
\phi(\nu,\eta)^{(k-2)2}:=1.
\end{align*}
These formulas should be compared with the formulas in \cite[Lemma 4.2]{PacRaeWill96}, 
on which they are based.

\section[Dimensional Reduction Isomorphisms]{Dimensional Reduction Isomorphisms}

Recall the projection $\pi^{0,3}$ of $\check{H}^3(X,\underline{\Z})$ onto the $E^{0,3}_\infty$ term 
of the Leray-Serre spectral sequence of $\pi:X\to Z$. Since
\[E^{0,3}_\infty=\{f\in C(Z,\check{H}^3(\R^n/\Z^n,\underline{\Z})):d_2f=d_3f=d_4f=0\},\]
we can view $\pi^{0,3}$ as a map $\pi^{0,3}:\check{H}^3(X,\underline{\Z})\to C(Z,\check{H}^3(\R^n/\Z^n,
\underline{\Z}))$, that can be calculated as follows. For each $x\in X$, let $\iota_x:\R^n/\Z^n\to X$ 
be the fibre inclusion $\iota_x:[t]\mapsto [-t]\cdot x$. Then, for $\delta\in\check{H}^3(X,\underline{\Z})$ 
we have
$(\pi^{0,3}\delta)(\pi(x))=\iota^*_x\delta.$
We will denote the kernel of $\pi^{0,3}$ by $\check{H}^3(X,\underline{\Z})|_{\pi^{0,3}=0}$. 
With this notation, we can state our main theorem, which we refer to as 
the \emph{dimensional reduction} isomorphisms.

\begin{theorem}\label{MainSquare}
Let $\pi:X\to Z$ be a $C^\infty$ principal $\T^n$-bundle over a Riemannian manifold $Z$. 
Then there exists a good open cover $\U$ of $X$, and a cochain $s\in \check{C}^1(\pi(\U),\mathcal{R}^n)$ 
such that
\begin{itemize}
\item[(1)] the image of $[\check\partial s]\in \check{H}^2(\pi(\U),\underline{\Z}^n)$ in 
$\check{H}^2(Z,\underline{\Z}^n)$ is the Euler vector of  $\pi:X\to Z$; and
\item[(2)] every stable continuous trace C*-algebra over $X$ is trivialised over $\U$.\\
Moreover, there is an isomorphism $\Xi_{\U,s}:\Br_{\R^n}(X)\to 
\mathbb{H}^2_{\check\partial s}(\pi(\U),\mathcal{S})$ and a commutative diagram \medskip

\centerline{\xymatrix{
\operatorname{Br}_{\R^n}(X)\ar[r]\ar[d]^{\cong}_{\Xi_{\U,s}}&\check{H}^3(X,\underline{\Z})|_{\pi^{0,3}=0}
\ar[d]^{\cong}\\
\mathbb{H}^2_{\check\partial s}(\pi(\U),\mathcal{S})\ar[r]&
\mathbb{H}^3_{\check\partial s}(\pi(\U),\underline{\Z})\,. 
}}
\end{itemize}
\end{theorem} \smallskip

The proof of this theorem is difficult and highly technical, with the most non-trivial part being 
the construction of $\Xi_{\U,s}$ and the proof that it is an isomorphism. In fact, 
injectivity of $\Xi_{\U,s}$ is reasonably straightforward; the main part is to show there is a 
commutative diagram\medskip

\centerline{\xymatrix{
\operatorname{Br}_{\R^n}(X)|_{M=0}\ar[r]^{\cong}\ar[dr]&H^2_{\R^n}(\pi^{-1}(\pi(\U)),\mathcal{S})\ar[d]\\
&\mathbb{H}^2_{\check\partial s}(\pi(\U),\mathcal{S})\,, }} \smallskip

\noindent where the downward arrow comes from Theorem 
\ref{PRWagreement}. For surjectivity on the other hand, one utilises the groupoid cohomology of Tu, 
which satisfies $\check{H}^2(G\ltimes X,\mathcal{S})\cong \Br_G(X)$ \cite[Cor 5.9]{Tu}, and
\cite{KumMuhRenWil98}, where $G\ltimes X$ is the transformation-group groupoid. 
Using this, we construct a map 
$\Psi_{\U,s}:\mathbb{H}^2_{\check\partial s}(\pi(\U),\mathcal{S}) \to \operatorname{Br}_{\R^n}(X)$ 
that factors through $\check{H}^2(G\ltimes X,\mathcal{S})$ and satisfies $\Xi_{\U,s}\circ\Psi_{\U,s}=
\operatorname{id}$ (see \cite{BouCarRat2} for details).

For the rest of this section we shall describe the isomorphisms $\Xi_{\U,s}$ and 
$\check{H}^3(X,\underline{\Z})|_{\pi^{0,3}=0}\to \mathbb{H}^3_{\check\partial s}(\pi(\U),\underline{\Z})$. 
This will be instructive in the following section, where we exhibit a concrete representative 
of every class in $\Br_{\R^n}(X)$ in terms of a cocycle in $Z^2_{\check\partial s}(\pi(\U),\mathcal{S})$.

\begin{definition}[\cite{Spi79}]
\begin{itemize}
\item[(i)] Let $(M,g)$ be a Riemannian manifold, and $U\subset M$. Then $U$ is said to be 
\emph{geodesically convex} if every two points $x,y\in U$ have a unique geodesic of 
minimum length between them, and that geodesic lies entirely within $U$.
\item[(ii)] Let $\U$ be an open cover of a Riemannian manifold $(M,g)$. Then we call 
$\U$ \emph{geodesically convex} if every open set in $\U$ is geodesically convex.
\end{itemize}
\end{definition}

Note that every geodesically convex open set is contractible, and every geodesically convex open 
cover is good \cite{Spi79}. It is instructive to see how the geometry is used for the following result. 
Thus we provide the proof for the first claim. The second claim then follows automatically, 
since the cover $\U$ is \emph{good}. To this end, let $\pi:X\to Z$ be a $C^\infty$ principal 
$\R^n/\Z^n$-bundle over a Riemannian manifold $Z$. We will show there is a good cover of 
$X$ that pushes down to a good cover of $Z$.

\begin{lemma}\label{geoconvex}
Let $(Z,dz^2)$ be a Riemannian manifold and $\pi:X\to Z$ a $C^\infty$ principal $\R^n/\Z^n$-bundle. 
Then $X$ has a geodesically convex open cover $\U=\{U_{\lambda}\}_{\lambda\in\mathcal{I}}$ 
such that $\pi(\U):=\{\pi(U_{\lambda})\}_{\lambda\in\mathcal{I}}$ is a 
geodesically convex open cover of  $Z$. 
Moreover, there exist $C^\infty$ local sections $\sigma_{\lambda}:\pi(U_\lambda)\to U_\lambda$.
\end{lemma}

\begin{proof}
Fix a connection 1-form $\omega$ on $\pi:X\to Z$, and let $\langle\ ,\ \rangle$ denote a bi-invariant 
metric on the Lie algebra $\mathfrak t$ of $\R^n/\Z^n$. Therefore we have a $\R^n$-bi-invariant 
Riemannian metric on $X$ defined by $g:=\pi^*dz^2+\langle\omega,\omega\rangle$. Moreover, 
with respect to the metrics $g$ and $dz^2$, the projection $\pi: X\to Z$ is a Riemannian submersion.

Now let $\U=\{U_{\lambda}\}$ be an open cover of $X$ consisting of 
geodesically convex sets defined as follows. For each $x\in X$ choose $\epsilon_0>0$ so that
$B_x(\epsilon_0)=\exp\{v\in TX_x:||v||<\epsilon_0\}$
is geodesically convex (this is possible by \cite[Vol. 1, Ex. 32, p.491]{Spi79}). 
We claim, perhaps after choosing a smaller $\epsilon_0$, that $\pi(\U)=\{\pi(U_\lambda)\}$ 
is a geodesically convex open cover of $Z$.

To see that this is the case, observe that \cite[Vol. 2, Ch. 8, Prop. 7]{Spi79} and \cite[Cor 1.1]{FalStePas04} 
show that a curve $\gamma$ in $Z$ is a geodesic in $Z$ if and only if its unique lift to a horizontal 
curve in $X$ is a geodesic. Moreover, \cite[Prop 1.10]{FalStePas04} says that, if $\gamma:I\to X$ 
is a geodesic such that $\dot{\gamma}(t_0)$ is horizontal at $x=\gamma(t_0)$, then 
$\gamma$ is horizontal. 
Finally, since $\pi$ is a \emph{Riemannian} submersion, the map $d\pi:TX_x\to TZ_{\pi(x)}$ is a 
surjection that preserves the length of horizontal vectors. Therefore, if $\epsilon_1>0$ is such that 
$B_{\epsilon_1}(\pi(x))$ is geodesically convex in $Z$, let $\epsilon=\min\{\epsilon_0,\epsilon_1\}$, 
and then $\pi(B_{\epsilon}(x))=B_{\epsilon}(\pi(x))$.

For the last claim, if $z_0\in B_{\epsilon}(\pi(x))$, then there exists a unique geodesic $\xi$ 
contained in $B_{\epsilon}(\pi(x))$ for which there exists $t\in [0,\epsilon)$ 
with $\xi(0)=\pi(x), \xi(t)=z_0$. From the above, there is a unique (horizontal) geodesic 
$\gamma$ in $B_{\epsilon}(x)$ such that $\pi(\gamma)=\xi$. Then we define 
$\sigma_{\lambda}(z_0):=\gamma(t)$.  This is a $C^\infty$ section by \cite[Vol 1, Thm 14(2)]{Spi79}.
\end{proof}
\begin{definition}
Let $\pi:X\to Z$ be a $C^\infty$ principal $\R^n/\Z^n$-bundle over a Riemannian 
manifold $Z$, and let $\U=\{U_{\lambda_0}\}_{\lambda_0\in\mathcal{I}}$ be a fixed 
good open cover of $X$ such that $\pi(\U)$ is a good open cover of $Z$. Also fix 
$C^\infty$ local sections $\sigma_{\lambda_0}:\pi(U_{\lambda_0})\to U_{\lambda_0}$ 
and a cochain $s\in \check{C}^1(\pi(\U),\mathcal{\R}^n)$ such that, for all indices 
$\lambda_0,\lambda_1\in\mathcal{I}$, and all $z\in \pi(U_{\lambda_0})\cap \pi(U_{\lambda_1})$, we have
$s_{\lambda_0\lambda_1}(z)\cdot\sigma_{\lambda_1}(z)=\sigma_{\lambda_0}(z),$
and
$s_{\lambda_1\lambda_1}(z)=0.$
Then we say the $(X,\U,s)$ is \emph{in the standard setup} (with the space $Z$, 
projection $\pi:X\to Z$ and local sections $\{\sigma_{\lambda_0}\}$ implicit).
\end{definition}

Now we proceed to the construction of the isomorphism $\Xi_{\U,s}$. Let $(X,\U,s)$ be 
in the standard setup, and fix $(CT(X,\delta),\alpha)\in \mathfrak{Br}_G(X)$. Thus there 
exist $C_0(U_{\lambda_0}^0)$-isomorphisms (local trivialisations)
$\Phi_{\lambda_0}:CT(X,\delta)|_{U_{\lambda_0}^0}\to C_0(U_{\lambda_0}^0,\K).$ We 
define a continuous map $\beta^{\alpha,\Phi}:(G\ltimes X[\U^0])^{(1)}\to\operatorname{Aut}\K$ 
as follows. For fixed $T\in\K$, let $a\in CT(X,\delta)$ satisfy $\Phi_{\lambda_0}(a)(g^{-1}x)=T$. 
Then we define
\[\beta^{\alpha,\Phi}_{(\lambda_0(g,x)\lambda_1)}(T):=\Phi_{\lambda_1}(\alpha_g(a))(x).\]
This is indeed well-defined:
\begin{lemma}\label{betawelledfined}
$\beta^{\alpha,\Phi}_{(\lambda_0(g,x)\lambda_1)}(T)$ is independent of the choice of $a\in CT(X,\delta)$ 
satisfying $\Phi_{\lambda_0}(a)(g^{-1}x)=T$.
\end{lemma}

\begin{proof}
Let $b$ be another element such that $\Phi_{\lambda_0}(b)(g^{-1}x)=T$. Then
$\Phi_{\lambda_0}(a-b)(g^{-1}x)=0.$
By \cite[Lemma 2.1]{EchWil98} and its proof, we can identify 
\begin{equation*}CT(X,\delta)|_{X\backslash\{g^{-1}x\}}
= C_0(X\backslash\{g^{-1}x\})\cdot CT(X,\delta)\,,
\end{equation*}
and there exists $f\in C_0(X)$ and $c\in CT(X,\delta)$ such that $f(g^{-1}x)=0$ and $f\cdot c=a-b$.
Then the calculation
\begin{align*}
\Phi_{\lambda_1}(\alpha_g(a-b))(x)& =\Phi_{\lambda_1}(\alpha_g(f\cdot c))(x)
=\Phi_{\lambda_1}(\tau_g(f)\alpha_g(c))(x) \\
& =f(g^{-1}x)\Phi_{\lambda_1}(\alpha_g(c))(x) =0 \,,
\end{align*}
shows that $\Phi_{\lambda_1}(\alpha_g(a))(x)=\Phi_{\lambda_1}(\alpha_g(b))(x)$.
\end{proof}

Now recall
\begin{lemma}[{\cite[Prop 4.27]{RaeWill98}}]\label{zetavanishremark}
Let $U$ be a contractible paracompact locally compact space, and $\beta:U\to\operatorname{Aut}\K$ 
a continuous map. Then there exists a continuous map $u:U\to U(\H)$ such that 
$\beta=\operatorname{Ad}u$.
\end{lemma}
We can thus introduce unitary valued lifts of $\beta^{\alpha,\Phi}$ by letting $\{e_i\}$ be generators 
of $\Z^n$ and recalling from \cite[Prop 2.1]{EchNes01} that, for the  trivialisations 
$\{\Phi_{\lambda_0}\}$, the maps $\Phi_{\lambda_0}\circ\alpha|_{\Z^n}\circ\Phi_{\lambda_0}^{-1}$ 
are locally inner. Thus, since $\U$ is good, {Lemma} \ref{zetavanishremark} implies there exist 
continuous maps $v_{\lambda_0}^i: U_{\lambda_0}\to U(\H)$ such that 
$\Phi_{\lambda_0}\circ\alpha_{e_i}\circ\Phi_{\lambda_0}^{-1}=\operatorname{Ad}v_{\lambda_0}^i$. 
Then we may define for any $m\in \Z^n$
\begin{equation}\label{vsplit}
v_{\lambda_0}^m(x):=(v_{\lambda_0}^1(x))^{m_1}(v_{\lambda_0}^2(x))^{m_2}\dots 
(v_{\lambda_0}^n(x))^{m_n},
\end{equation}
which gives
\begin{equation}\label{betav01}
\beta^{\alpha,\Phi}_{(\lambda_0(m,x)\lambda_0)}(T)=\operatorname{Ad}v_{\lambda_0}^m(x)(T).
\end{equation}
There is also a map from $\pi(U_{\lambda_0})\cap \pi(U_{\lambda_1})\to \operatorname{Aut}\K$ given by
\begin{equation*}
z\mapsto \beta^{\alpha,\Phi}_{(\lambda_0(-s_{\lambda_0\lambda_1}(z),
\sigma_{\lambda_1}(z))\lambda_1)}\,.
\end{equation*}
Since $\pi(\U)$ is a good cover of $Z$, the open set $\pi(U_{\lambda_0})\cap \pi(U_{\lambda_1})$ 
is contractible, and there exists a continuous map $v_{\lambda_0\lambda_1}:
\pi(U_{\lambda_0})\cap \pi(U_{\lambda_1})\to U(\H)$ that satisfies
\begin{equation}\label{betav10}
\beta^{\alpha,\Phi}_{(\lambda_0(-s_{\lambda_0\lambda_1}(z),\sigma_{\lambda_1}(z))\lambda_1)}
=\operatorname{Ad}v_{\lambda_0\lambda_1}(z).
\end{equation}

Next, $\Xi_{\U,s}[CT(X,\delta),\alpha]$ will be a triple $[\phi(\alpha)^{20},
\phi(\alpha)^{11},\phi(\alpha)^{02}]$.
Define the last component $\phi(\alpha)^{02}\in \check{C}^0(\pi(\U),\mathcal{M})$ by:
\begin{equation}\label{phi02defn}
\phi(\alpha)^{02}_{\lambda_0}(z)_{ij}:=v_{\lambda_0}^j(\sigma_{\lambda_0}
(z))v_{\lambda_0}^i(\sigma_{\lambda_0}(z))v_{\lambda_0}^j(\sigma_{\lambda_0}
(z))^*v_{\lambda_0}^i(\sigma_{\lambda_0}(z))^*.
\end{equation}
(This equation should be compared with Eqn.~(\ref{MackeyDefn})). 
For the middle component, note that since $\alpha$ is a homomorphism 
and $\R^n$ is abelian we have
\[
\operatorname{Ad} \left[v_{\lambda_0\lambda_1}(z)v_{\lambda_0}^m(z)\right]
=\operatorname{Ad}\left[ v_{\lambda_1}^m(z)v_{\lambda_0\lambda_1}(z)\right].
\]
A computation then shows that the map
$m\mapsto v_{\lambda_1}^m(z)v_{\lambda_0\lambda_1}
(z)v_{\lambda_0}^m(z)^*v_{\lambda_0\lambda_1}(z)^*$
is a homomorphism from $\Z^n$ to $\T$. This being the case, 
we may define $\phi(\alpha)^{11}\in\check{C}^1(\pi(\U),\hat{{\mathcal N}})$ by
\begin{equation}\label{phi11defn}
\phi(\alpha)^{11}_{\lambda_0\lambda_1}(m,z):=v_{\lambda_1}^m(z)
v_{\lambda_0\lambda_1}(z)v_{\lambda_0}^m(z)^*v_{\lambda_0\lambda_1}(z)^*.
\end{equation}
Finally,  to define the  third 
component, we use a similar calculation to the previous one to show
\[\operatorname{Ad}  \left[v_{\lambda_1\lambda_2}(z)
v_{\lambda_0\lambda_1}(z)\right]=\operatorname{Ad} 
\left[ v_{\lambda_0\lambda_2}(z)v_{\lambda_0}^{-\check\partial 
s_{\lambda_0\lambda_1\lambda_2}(z)}(z)\right] .
\]
Therefore we may define the cochain $\phi(\alpha)^{20}\in\check{C}^2(\pi(\U),{\mathcal S})$ by
\begin{equation}\label{phi20defn}
\phi(\alpha)^{20}_{\lambda_0\lambda_1\lambda_2}(z):=
v_{\lambda_1\lambda_2}(z)v_{\lambda_0\lambda_1}(z)v_{\lambda_0}^{-\check\partial 
s_{\lambda_0\lambda_1\lambda_2}(z)}(z)^*v_{\lambda_0\lambda_2}(z)^*.
\end{equation}
It is then immediate that the triple $(\phi(\alpha)^{20},\phi(\alpha)^{11},\phi(\alpha)^{02})$ 
is a cochain in $C^2_{\check\partial s}(\pi(\U),{\mathcal S})$, and with some difficulty one 
can show that it is $D_{\check\partial s}$-closed \cite{BouCarRat2}. The map $\Xi_{\U,s}$ is then defined by
\begin{equation}\label{Xi}\Xi_{\U,s}:(CT(X,\delta),\alpha)\mapsto[\phi(\alpha)^{20},
\phi(\alpha)^{11},\phi(\alpha)^{02}].\end{equation}
With a further difficult computation, one can show that $\Xi_{\U,s}$ is constant 
on outer conjugacy classes, and hence defines a map with domain $\Br_{\R^n}(X)$.

To complete the discussion of Theorem \ref{MainSquare}, we need to formulate the right downward arrow
$\check{H}^3(X,\underline{\Z})|_{\pi^{0,3}=0}\stackrel{\cong}{\longrightarrow}
{\mathbb H}^3_{\check\partial s}(\pi(\U),\underline{\Z})$,
and  commutativity of the diagram. This arrow is only defined implicitly by chasing around a diagram. 
This is unfortunate, because an explicit map may allow generalisation to a dimensional 
reduction isomorphism from all of $\check{H}^3(X,\underline{\Z})$ to a four-column 
complex generalising $\mathbb{H}^3_{\check\partial s}(\pi(\U),\mathcal{S})$. 
The first step is thus to explain the purpose of the restriction to 
$\check{H}^3(X,\underline{\Z})|_{\pi^{0,3}=0}$.

\begin{theorem}[{\cite[Thm 2.2]{MatRos06}}]
Let $\delta\in\check{H}^3(X,\underline{\Z})$. Then $\delta$ is in the image of the forgetful homomorphism
$F:\operatorname{Br}_{\R^n}(X)\to \check{H}^3(X,\underline{\Z})$
if and only if $\pi^{0,3}(\delta)=0$.
\end{theorem}
Next, as every class in $\check{H}^3(X,\underline{\Z})|_{\pi^{0,3}=0}$ lifts to 
$\mbox{Br}_{\R^n}(X)$ we can implicitly define a map  from 
$\check{H}^3(X,\underline{\Z})|_{\pi^{0,3}=0}$ to 
${\mathbb H}^3_{\check\partial s}(\pi(\U),\underline{\Z})$ as the composition

\centerline{\xymatrix{
\mathrm{Br}_{\R^n}(X)\ar[r]^{\Xi_{\U,s}}&{\mathbb H}^2_{\check\partial s}(\pi(\U),{\mathcal S})\ar[d]\\
\check{H}^3(X,\underline{\Z})|_{\pi^{0,3}=0}\ar[u]^{F^{-1}}&
{\mathbb H}^3_{\check\partial s}(\pi(\U),\underline{\Z}),
}}

\noindent provided the image in ${\mathbb H}^3_{\check\partial s}(\pi(\U),\underline{\Z})$ 
is independent of the choice of the lift\\ $\check{H}^3(X,\underline{\Z})
|_{\pi^{0,3}=0}\to \mathrm{Br}_{\R^n}(X)$. 
{}From  {\cite[Thm 2.3]{MatRos06}} we have
\begin{proposition}\label{BrauerKernel}
There is an exact sequence
\[H^2_M(\R^n,C(X,\T))\to \mathrm{Br}_{\R^n}(X)\to \check{H}^3(X,\underline{\Z})\vline_{\pi^{0,3}=0}\to 0.\]
\end{proposition}
To find the image of $H^2_M(\R^n,C(X,\T))$ under the composition
\[H^2_M(\R^n,C(X,\T))\to \mathrm{Br}_{\R^n}(X)\stackrel{\Xi_{\U,s}}{\longrightarrow}
{\mathbb H}^2_{\check\partial s}(\pi(\U),{\mathcal S}),\]
we now describe $H^2_M(\R^n,C(X,\T))\to \mathrm{Br}_{\R^n}(X)$ in more detail. 
First, we know from \cite[Lemma 2.1]{MatRos05} that there is an isomorphism 
$C(Z,M_n^u(\R)) \cong H^2_M(\R^n,C(X,\T))$, taking $g\in C(Z,M_n^u(\R))$ to  the 
cocycle $\tilde{g}$ in $Z^2_M(\R^n,C(X,\T))$, defined by
\[(s,t)\mapsto \left(x\mapsto \left[\sum_{1\leq i<j\leq n}g(\pi(x))_{ij}t_is_j\right]_{\R/\Z}\right).\]
We find the image of this in $\mathrm{Br}_{\R^n}(X)$ as follows. Let $\H=L^2(\R^n)$, 
and define a continuous map $L_{\tilde{g}}:\R^n\to C(X,U(L^2(\R^n)))$ with the formula
\[[L_{\tilde{g}}(s)(x)](\xi)(t):=\tilde{g}(s,t-s)(x)\xi(t-s),\quad s,t\in\R^n, x\in X, \xi\in L^2(\R^n).\]
This satisfies
\[
L_{\tilde{g}}(s)(x)[L_{\tilde{g}}(t)(-s\cdot x)\xi](r)=\tilde{g}(s,t)(x)[L_{\tilde{g}}(s+t)(x)\xi](r).\]
It follows that the map $H^2_M(\R^n,C(X,\T))\cong C(Z,M_n^u(\R))\to 
\mathrm{Br}_{\R^n}(X)$ is given (cf. \cite[Thm 5.1]{CroKumRaeWil97}) by
$g\mapsto [C_0(X,\K),\operatorname{Ad}L_{\tilde{g}}\circ \tau].$

\begin{lemma}\label{Liftindependence}
Let $(X,\U,s)$ be in the standard setup. Then the composition
$$C(Z,M_n^u(\R))\rightarrow \mathrm{Br}_{\R^n}(X)\stackrel{\Xi_{\U,s}}{\rightarrow}
 {\mathbb H}^2_{\check\partial s}(\pi(\U),{\mathcal S})\rightarrow
 {\mathbb H}^3_{\check\partial s}(\pi(\U),\underline{\Z})$$
\noindent is the zero map.
\end{lemma}

\noindent We include a proof for this lemma to help the reader understand Section \ref{UniversalSection}.

\begin{proof}[of Lemma \ref{Liftindependence}]
Let $g\in C(Z,M_n^u(\R))$ have image $(C_0(X,\K),\operatorname{Ad}L_{\tilde{g}}\circ \tau)
\in\mathfrak{Br}_{\R^n}(X)$. Then, if we take the identity maps as the local trivialisations, we have
\begin{align*}
\beta^{\alpha,\Phi}_{(\lambda_0(m,\sigma_{\lambda_0}(z))\lambda_0)} & =
\operatorname{Ad}L_{\tilde{g}(\cdot,\cdot)(\sigma_{\lambda_0}(z))}(m), 
  \mbox{
and}\quad \\
\beta^{\alpha,\Phi}_{(\lambda_0(-s_{\lambda_0\lambda_1}(z),\sigma_{\lambda_1}(z))\lambda_1)}
& =\operatorname{Ad}L_{\tilde{g}(\cdot,\cdot)(\sigma_{\lambda_1}(z))}(-s_{\lambda_0\lambda_1}(z))\,.
\end{align*}
Therefore we can define
\begin{align*}
u_{\lambda_0}^m(z):=&(L_{\tilde{g}(\cdot,\cdot)(\sigma_{\lambda_0}(z))}(e_1))^{m_1}
\dots(L_{\tilde{g}(\cdot,\cdot)(\sigma_{\lambda_0}(z))}(e_n))^{m_n} \quad\mbox{and}\\
u_{\lambda_0\lambda_1}(z):=&L_{\tilde{g}(\cdot,\cdot)(\sigma_{\lambda_1}(z))}
(-s_{\lambda_0\lambda_1}(z)).
\end{align*}
We can then calculate the image $[\phi(\operatorname{Ad}L\circ\tau)^{20},\phi
(\operatorname{Ad}L\circ\tau)^{11},\phi(\operatorname{Ad}L\circ\tau)^{02}]$ of 
$(C_0(X,\K),\operatorname{Ad}L_{\tilde{g}}\circ \tau)$ in ${\mathbb H}^2_{\check\partial s}
(\pi(\U),{\mathcal S})$ using
\begin{equation}\label{Lequation}
L_{\tilde{g}}(s)(x)L_{\tilde{g}}(t)(-s\cdot x)=\tilde{g}(s,t)(x)L_{\tilde{g}}(s+t)(x).
\end{equation}
as follows. By the definitions from Eqns.~(\ref{phi02defn}) and (\ref{phi11defn}), 
the last two components are
\begin{align*}
\phi(\operatorname{Ad}L\circ\tau)^{02}_{\lambda_0}(z)_{ij}
&\quad=L_{\tilde{g}(\cdot,\cdot)(\sigma_{\lambda_0}(z))}(e_j)
L_{\tilde{g}(\cdot,\cdot)(\sigma_{\lambda_0}(z))}(e_i) \\
& \quad \quad \times 
L_{\tilde{g}(\cdot,\cdot)(\sigma_{\lambda_0}(z))}(e_j)^*
L_{\tilde{g}(\cdot,\cdot)(\sigma_{\lambda_0}(z))}(e_i)^*\\
&\quad=\tilde{g}(e_j,e_i)(\sigma_{\lambda_0}(z))\quad=\quad [g(z)_{ij}]_{\R/\Z},\quad\quad\\
\phi(\operatorname{Ad}L\circ\tau)^{11}_{\lambda_0\lambda_1}(m,z)
&\quad=u_{\lambda_1}^m(z)u _{\lambda_0\lambda_1}(z)u_{\lambda_0}^m(z)^*
u_{\lambda_0\lambda_1}(z)^* \quad\mbox{ by } (\ref{Lequation})\times 2\\
&\quad=L_{\tilde{g}(\cdot,\cdot)(\sigma_{\lambda_1}(z))}(m)
L_{\tilde{g}(\cdot,\cdot)(\sigma_{\lambda_1}(z))}(-s_{\lambda_0\lambda_1}(z))\\
&\quad\quad\times L_{\tilde{g}(\cdot,\cdot)(\sigma_{\lambda_0}(z))}(m)^*
L_{\tilde{g}(\cdot,\cdot)(\sigma_{\lambda_1}(z))}(-s_{\lambda_0\lambda_1}(z))^*\\
&\quad=\tilde{g}(m,-s_{\lambda_0\lambda_1}(z))(\sigma_{\lambda_1}(z))\tilde{g}
(-s_{\lambda_0\lambda_1}(z),m)(\sigma_{\lambda_1}(z))^*\\
&\quad=\left[\sum_{1\leq i<j\leq n}g(z)_{ij}(m_is_{\lambda_0\lambda_1}(z)_j-
s_{\lambda_0\lambda_1}(z)_im_j)\right]_{\R/\Z},\quad\quad\quad\quad
\end{align*}
The last component is computed using Eqn.~(\ref{phi20defn}):
\begin{align*}
\phi & (\operatorname{Ad}L\circ\tau)^{20}_{\lambda_0\lambda_1\lambda_2}(z)
= u _{\lambda_1\lambda_2}(z)u _{\lambda_0\lambda_1}(z)
u_{\lambda_0}^{-\check\partial s_{\lambda_0\lambda_1\lambda_2}(z)}(z)
u _{\lambda_0\lambda_2}(z)^*\\
&\quad=L_{\tilde{g}(\cdot,\cdot)(\sigma_{\lambda_2}(z))}(-s_{\lambda_1\lambda_2}(z))
L_{\tilde{g}(\cdot,\cdot)(\sigma_{\lambda_1}(z))}(-s_{\lambda_0\lambda_1}(z))\\
&\quad\quad\times L_{\tilde{g}(\cdot,\cdot)(\sigma_{\lambda_0}(z))}
(-\check\partial s_{\lambda_0\lambda_1\lambda_2}(z))^*L_{\tilde{g}(\cdot,\cdot)
(\sigma_{\lambda_2}(z))}(-s_{\lambda_0\lambda_2}(z))^*\\
&\quad=\tilde{g}(-s_{\lambda_1\lambda_2}(z),-s_{\lambda_0\lambda_1}(z))\tilde{g}
(-s_{\lambda_0\lambda_2}(z),-\check\partial s_{\lambda_0\lambda_1\lambda_2}(z))^*\\
&\quad=\left[\sum_{1\leq i<j\leq n}g(z)_{ij}(s_{\lambda_0\lambda_1}(z)_i
s_{\lambda_1\lambda_2}(z)_j-\check\partial s_{\lambda_0\lambda_1\lambda_2}
(z)_is_{\lambda_0\lambda_2}(z)_j)\right]_{\R/\Z}.
\end{align*}

The map from ${\mathbb H}^2_{\check\partial s}(\pi(\U),{\mathcal S})$ to 
${\mathbb H}^3_{\check\partial s}(\pi(\U),\underline{\Z})$ is just a standard 
``connecting homomorphism", computing using a zig-zag:

\smallskip

\centerline{\xymatrix{
Z^3_{\check\partial s}(\pi(\U),\underline{\Z})&Z^3_{\check\partial s}(\pi(\U),
{\mathcal R})\ar[l]& \\
&C^2_{\check\partial s}(\pi(\U),{\mathcal R})\ar[u]^{D_{\check\partial s}}&
Z^2_{\check\partial s}(\pi(\U),{\mathcal S})\ar[l]
}}

\smallskip

\noindent Thus, the image of $[\phi(\operatorname{Ad}L\circ\tau)^{20},
\phi(\operatorname{Ad}L\circ\tau)^{11},\phi(\operatorname{Ad}L\circ\tau)^{02}]$ 
in ${\mathbb H}^3_{\check\partial s}(\pi(\U),\underline{\Z})$ is a triple
\begin{align*}
&\big[\Delta(\phi(\operatorname{Ad}L\tau)^{20},
\phi(\operatorname{Ad}L\tau)^{11},\phi(\operatorname{Ad}L\tau)^{02})^{30}, \\
&\quad \Delta(\phi(\operatorname{Ad}L\tau)^{20},
\phi(\operatorname{Ad}L\tau)^{11},\phi(\operatorname{Ad}L\tau)^{02})^{21},\\
&\quad\quad\Delta(\phi(\operatorname{Ad}L\tau)^{21},
\phi(\operatorname{Ad}L\tau)^{11},\phi(\operatorname{Ad}L\tau)^{02})^{12}\big].
\end{align*}
The latter two terms are relatively easy to compute:
\begin{align*}
&\Delta(\phi(\operatorname{Ad}L\tau)^{20},\phi(\operatorname{Ad}L\tau)^{11},
\phi(\operatorname{Ad}L\tau)^{02})_{\lambda_0\lambda_1\lambda_2}^{21}(z)_l\\
&\quad=\check{\partial}\left[\sum_{1\leq i<j\leq n}g_{ij}(z)((e_l)_is_{\cdot\cdot}(z)_j-
s_{\cdot\cdot}(z)_i(e_l)_j)\right]_{_{\lambda_0\lambda_1\lambda_2}}\\
& \quad\quad\quad +\sum_{1\leq i<j\leq n} g(z)_{ij}(F_{\lambda_0\lambda_1
\lambda_2}(z)_i(e_l)_j-(e_l)_iF_{\lambda_0\lambda_1\lambda_2}(z)_j)=0.
\end{align*}
$$\Delta(\phi(\operatorname{Ad}L\tau)^{20},\phi(\operatorname{Ad}L\tau)^{11},
\phi(\operatorname{Ad}L\tau)^{02})_{\lambda_0\lambda_1}^{12}(z)_{ij}=g(z)_{ij}-g(z)_{ij}=0.$$
The first term is similar, but requires one to compute the \v Cech differential $\check\partial$ of
\[\bigcap_{k=0}^2 \pi(U_{\lambda_k})\ni z\mapsto \sum_{1\leq i<j\leq n}g(z)_{ij}
(s_{\lambda_0\lambda_1}(z)_is_{\lambda_1\lambda_2}(z)_j-\check\partial 
s_{\lambda_0\lambda_1\lambda_2}(z)_is_{\lambda_0\lambda_2}(z)_j).\]
We find this differential has the formula
\begin{align*}\bigcap_{k=0}^3 \pi(U_{\lambda_k})\ni z\mapsto &
\sum_{1\leq i<j\leq n}g(z)_{ij}\left[\check\partial 
s_{\lambda_0\lambda_1\lambda_2}(z)_i\check\partial 
s_{\lambda_0\lambda_2\lambda_3}(z)_j-\check\partial 
s_{\lambda_1\lambda_2\lambda_3}(z)_i\check\partial 
s_{\lambda_0\lambda_1\lambda_3}(z)_j\right.\\
&\quad \left. \check\partial s_{\lambda_1\lambda_2\lambda_3}(z)_i 
s_{\lambda_0\lambda_1}(z)_j-s_{\lambda_0\lambda_1}(z)_i\check\partial 
s_{\lambda_1\lambda_2\lambda_3}(z)_j\right].
\end{align*}
Then we see that
\begin{align*}
&\Delta(\phi(\operatorname{Ad}L\tau)^{20},\phi(\operatorname{Ad}L\tau)^{11},
\phi(\operatorname{Ad}L\tau)^{02})_{\lambda_0\lambda_1\lambda_2\lambda_3}^{30}(z)\\
=&\sum_{1\leq i<j\leq n}g(z)_{ij}\left[\check\partial s_{\lambda_0\lambda_1\lambda_2}(z)_i
\check\partial s_{\lambda_0\lambda_2\lambda_3}(z)_j-\check\partial 
s_{\lambda_1\lambda_2\lambda_3}(z)_i\check\partial 
s_{\lambda_0\lambda_1\lambda_3}(z)_j\right.\\
&\quad \left. +\check\partial s_{\lambda_1\lambda_2\lambda_3}(z)_i s_{\lambda_0
\lambda_1}(z)_j-s_{\lambda_0\lambda_1}(z)_i\check\partial 
s_{\lambda_1\lambda_2\lambda_3}(z)_j\right]\\
&\quad-\sum_{1\leq i<j\leq n}g(z)_{ij}(\check\partial s_{\lambda_1\lambda_2\lambda_3}
(z)_is_{\lambda_0\lambda_1}(z)_j-s_{\lambda_0\lambda_1}(z)_i\check\partial
 s_{\lambda_1\lambda_2\lambda_3}(z)_j)\\
&\quad-\sum_{1\leq i<j\leq n}g(z)_{ij}(\check\partial s_{\lambda_0\lambda_1
\lambda_2}(z)_i\check\partial s_{\lambda_0\lambda_2\lambda_3}(z)_j-\check\partial
 s_{\lambda_1\lambda_2\lambda_3}(z)_i\check\partial s_{\lambda_0\lambda_1\lambda_3}(z)_j)
\quad =0.
\end{align*}
\end{proof}

Applying Proposition \ref{BrauerKernel} and Lemma \ref{Liftindependence}
\begin{corollary}\label{NonclassicalLiftIndependence}
Let $(X,\U,s)$ be in the standard setup, and suppose 
$\delta\in \check{H}^3(X,\underline{\Z})|_{\pi^{0,3}=0}$. Then the image of 
$\delta$ under the composition

\centerline{\xymatrix{
\mathrm{Br}_{\R^n}(X)\ar[r]^{\Xi_{\U,s}}&{\mathbb H}^2_{\check\partial s}(\pi(\U),{\mathcal S})\ar[d]&\\
\check{H}^3(X,\underline{\Z})|_{\pi^{0,3}=0}\ar[u]^{F^{-1}}&
{\mathbb H}^3_{\check\partial s}(\pi(\U),\underline{\Z})
}}

\noindent is independent of the choice of lift 
$\check{H}^3(X,\underline{\Z})|_{\pi^{0,3}=0}\to \mathrm{Br}_{\R^n}(X)$.
\end{corollary}

One can then prove that $\check{H}^3(X,\underline{\Z})|_{\pi^{0,3}=0}\to
{\mathbb H}^3_{\check\partial s}(\pi(\U),\underline{\Z})$ is an isomorphism 
using the following commutative diagram with exact rows and an application of the 
Five Lemma:

\centerline{\xymatrix{
C(Z,M_n^u(\R))\ar[r]\ar[d]^{\operatorname{id}}&\mathrm{Br}_{\R^n}(X)
\ar[d]^{\cong}\ar[r]&\check{H}^3(X,\underline{\Z})|_{\pi^{0,3}=0}\ar[r]\ar[d]& 0\ar[r]\ar[d]&0\ar[d]\\
C(Z,M_n^u(\R))\ar[r]&{\mathbb H}^2_{\check\partial s}(\pi(\U),{\mathcal S})
\ar[r]&{\mathbb H}^3_{\check\partial s}(\pi(\U),\underline{\Z})\ar[r]& 0\ar[r]& 0
}}

\noindent Now Theorem \ref{MainSquare} follows immediately.

\section{A Universal Example}\label{UniversalSection}

We begin by introducing some notation.

\begin{lemma}\label{Rnvaluedlifts}
Let $(X,\U,s)$ be in the standard setup. Define $\R^n/\Z^n$-equivariant functions 
$w_{\lambda_0}:\pi^{-1}(\pi(U_{\lambda_0}))\to \R^n/\Z^n$ by $w_{\lambda_0}(x)^{-1}x:=
\sigma_{\lambda_0}(\pi(x))$. Then for all indices ${\lambda_0}$ there is a unique continuous function $
\tilde{w}_{\lambda_0}:U_{\lambda_0}\to \R^n$ such that for all $x\in U_{\lambda_0}$,
$[\tilde{w}_{\lambda_0}(x)]_{\R^n/\Z^n}=w_{\lambda_0}(x)$, and $\tilde{w}_{\lambda_0}
(\sigma_{\lambda_0}(\pi(x)))=0\in\R^n$.
\end{lemma}

\begin{proof}
Consider the exact sequence
\[C(U_{\lambda_0},\Z^n)\to C(U_{\lambda_0},\R^n)\to C(U_{\lambda_0},\T^n)\to
 \check{H}^1(U_{\lambda_0},\Z^n).\]
Then $w_{\lambda_0}|_{U_{\lambda_0}}$ has a lift to a continuous function $
\tilde{w}_{\lambda_0}:U_{\lambda_0}\to \R^n$ if and only if the image of $w_{\lambda_0}$ 
under the connecting homomorphism
$C(U_{\lambda_0},\T^n)\to \check{H}^1(U_{\lambda_0},\Z^n)$
is trivial. But, we have a locally constant sheaf over a contractible space $U_{\lambda_0}$  
because $\U$ is good, and thus we conclude $\tilde{w}_{\lambda_0}$ exists.

For uniqueness and the last condition, we just pick any $\tilde{w}_{\lambda_0}$ and 
consider $x\mapsto\tilde{w}_{\lambda_0}(x)-\tilde{w}_{\lambda_0}(\sigma_{\lambda_0}(\pi(x)))$.
\end{proof}

Suppose that $(X,\mathcal{U},s)$ is in the ``standard setup", and let 
$(A,\alpha)\in\mathfrak{Br}_{\R^n}(X)$ have image $[\phi]=[\phi^{20},\phi^{11},\phi^{02}]\in
 \mathbb{H}^2_{\partial s}(\pi(\U),\mathcal{S})$ under $\Xi_{\U,s}$. 
By the construction of $\Xi_{\U,s}$, we may suppose that there exist maps
\begin{equation}\label{damnunitaries}v_{\lambda_0}^\bullet:\,\Z^n\to 
C(\pi(U_{\lambda_0}),U(\H)), \quad
v_{\lambda_0\lambda_1}\in C(\pi(U_{\lambda_0})\cap \pi(U_{\lambda_1}),U(\H))
\end{equation}
such that $\beta^{\alpha,\Phi}_{(\lambda_0(-s_{\lambda_0\lambda_1}(z),z)\lambda_1)}=
\operatorname{Ad}v_{\lambda_0\lambda_1}(z)$, 
$\beta^{\alpha,\Phi}_{(\lambda_0(m,z)\lambda_0)}=\operatorname{Ad}v_{\lambda_0}^m(z)$ and
\begin{align}
\label{eq1}\phi^{02}_{\lambda_0}(z)_{ij}=&v_{\lambda_0}^{e_j}(\sigma_{\lambda_0}
(z))v_{\lambda_0}^{e_i}(\sigma_{\lambda_0}(z))v_{\lambda_0}^{e_j}(\sigma_{\lambda_0}
(z))^*v_{\lambda_0}^{e_i}(\sigma_{\lambda_0}(z))^*,\\
\label{eq2}\phi^{11}_{\lambda_0\lambda_1}(m,z)=&v_{\lambda_1}^m(z)
v_{\lambda_0\lambda_1}(z)v_{\lambda_0}^m(z)^*v_{\lambda_0\lambda_1}(z)^*,\\
\label{eq3}\phi^{20}_{\lambda_0\lambda_1\lambda_2}(z)=&v_{\lambda_1\lambda_2}
(z)v_{\lambda_0\lambda_1}(z)v_{\lambda_0}^{-\check\partial 
s_{\lambda_0\lambda_1\lambda_2}(z)}(z)^*v_{\lambda_0\lambda_2}(z)^*.
\end{align}
We define a C*-algebra $A_{\phi}$ as
\begin{align*}
A_{\phi}:=&\big\{\xi=\{\xi_{\lambda_0}\}:\xi_{\lambda_0}\in\operatorname{Ind}_{\Z^n}^{\R^n}
(C_0(\pi(U_{\lambda_0}),\K),\operatorname{Ad}v_{\lambda_0}^\bullet) \mbox{ and}\\
& \operatorname{Ad}v_{\lambda_0\lambda_1}(z)\big[\xi_{\lambda_0}(t)(z)\big]=
\xi_{\lambda_1}(t+s_{\lambda_0\lambda_1}(z))(z), \\ & \quad \forall z\in \pi(U_{\lambda_0})\cap 
\pi(U_{\lambda_1}),t\in\R^n\big\}.
\end{align*}
A computation shows this algebra is well-defined. Moreover, it comes with an action of $\R^n$ given by
\[\alpha'_s(\xi)_{\lambda_0}(t)(z):= \xi_{\lambda_0}(t-s)(z),\]
and we claim that $\Xi_{\U,s}[A_{\phi},\alpha']=[\phi^{20},\phi^{11},\phi^{02}]$. 
For the next proposition, we will realise $X$ via the isomorphism $X\cong 
\coprod_{\lambda_0}\pi(U_{\lambda_0})\times \T^n /\sim$, where $\sim$ is the 
equivalence relation
\begin{equation}\label{Xequivalence}
(\lambda_0,z,[t])\sim (\lambda_1,z,[t+s_{\lambda_0\lambda_1}(z)])\quad 
\mbox{ for all } z\in \pi(U_{\lambda_0})\cap \pi(U_{\lambda_1}).
\end{equation}

\begin{proposition}
Let $\{\xi_{\lambda_0}\}\in A_{\phi}$, and define $M(t,\lambda_0,z):\xi\mapsto 
\xi_{\lambda_0}(t)(z)$. Then the map $(\lambda_0,z,t)\mapsto M(\lambda_0,z,t)$ 
induces a homeomorphism of $X$ onto $\hat{A_\phi}$.
\end{proposition}
\begin{proof}
By \cite[Prop 6.16]{RaeWill98} we have a homeomorphism
\[M:\coprod_{\lambda_0}\pi(U_{\lambda_0})\times \T^n\to \coprod_{\lambda_0} 
\mbox{Ind}_{\Z^n}^{\R^n}(C_0(\pi(U_{\lambda_0}),\K),\operatorname{Ad}v_{\lambda_0}^\bullet)\,
\widehat{{}}\]
It therefore suffices to show that the representation $M(\lambda_0,z,t)$ is unitarily equivalent to
 $M(\lambda_1,z,t+s_{\lambda_0\lambda_1}(z))$, but this follows from the definition of $A_\phi$.
\end{proof}

\begin{proposition}
The image of $[A_{\phi},\alpha']$ under  $\Xi_{\U,s}$ equals $[\phi^{20},\phi^{11},\phi^{02}]$, and 
$(A,\alpha)$ is outer conjugate to $(A_{\phi},\alpha')$.
\end{proposition}
\begin{proof}
We have local isomorphisms $\Phi_{\lambda_0}:A_{\phi}|_{U_{\lambda_0}}\to 
C_0(U_{\lambda_0},\K)$ given by
\[\Phi_{\lambda_0}(\xi)(x):=\xi_{\lambda_0}(\tilde{w}_{\lambda_0}(x))(\pi(x)).\]
Fix $T\in \K$ and let $\xi$ be any element such that $T=\Phi_{\lambda_0}(\xi)
(\sigma_{\lambda_0}(z))=\xi_{\lambda_0}(\tilde{w}_{\lambda_0}
(\sigma_{\lambda_0}(z)))(z)=\xi_{\lambda_0}(0)(z)$. 
Observe
\begin{align*}
\beta^{\alpha,\Phi}_{\lambda_0(m,\sigma_{\lambda_0}(z))\lambda_0}(T)=&
\Phi_{\lambda_0}(\alpha_m\xi)(\sigma_{\lambda_0}(z))=
(\alpha_m\xi)_{\lambda_0}(\tilde{w}_{\lambda_0}(\sigma_{\lambda_0}(z)))(z)\\
=&\xi_{\lambda_0}(0-m)(z)=\operatorname{Ad}v_{\lambda_0}^m[\xi_{\lambda_0}(0)(z)]=
\operatorname{Ad}v_{\lambda_0}^m[T].
\end{align*}
Also, if $
T=\Phi_{\lambda_0}(\xi)(s_{\lambda_0\lambda_1}(\pi(x))\cdot \sigma_{\lambda_1}(z))
=\xi_{\lambda_0}(0)(z)
$,
we have
\begin{align*}
&\beta^{\alpha,\Phi}_{\lambda_0(-s_{\lambda_0\lambda_1}(z),\sigma_{\lambda_1}(z))\lambda_1}(T)
=\Phi_{\lambda_1}(\alpha_{-s_{\lambda_0\lambda_1}(z)}\xi)(\sigma_{\lambda_1}(z))
\\ \quad =& (\alpha_{-s_{\lambda_0\lambda_1}(z)}\xi)_{\lambda_1}(0)(z)
= \xi_{\lambda_1}(s_{\lambda_0\lambda_1}(z))(z)=\operatorname{Ad}v_{\lambda_0\lambda_1}(z)
\big[\xi_{\lambda_0}(0)(z)\big] \\ \quad = & \operatorname{Ad}v_{\lambda_0\lambda_1}(z)[T].
\end{align*}
It follows immediately from the construction of $\Xi_{\U,s}$ that 
$\Xi_{\U,s}[A_{\phi},\alpha']=[\phi^{20},\phi^{11},\phi^{02}]$, and the final statement of the 
proposition is true since $\Xi_{\U,s}$ is injective.

\end{proof}

\begin{remark}
Note that the construction of $A_{\phi}$ requires the existence of unitary valued maps 
satisfying Equations (\ref{eq1})-(\ref{eq3}), which are known to exists by a difficult and 
indirect argument showing $\Xi_{\U,s}$ is surjective (see
\cite{BouCarRat2}). We are as yet unaware of a direct proof.
\end{remark}

\begin{acknowledgement}
PB and AC would like to thank the organizers of the RIMS International Conference on
Noncommutative Geometry and Physics for the invitation and hospitality.
This research was supported under Australian Research Council's Discovery
Projects funding scheme (project numbers DP0559415 and DP0878184).
\end{acknowledgement}



\printindex
\end{document}